\documentclass{article}

\usepackage{delarray}
\usepackage{amsmath, latexsym, amsfonts, amssymb, amsthm, amscd}
\usepackage{graphicx}
\usepackage{graphicx,,psfrag}
\usepackage{amssymb,amsmath,amsthm,enumerate}
\usepackage{color}
\fontsize{10}{13}\selectfont 
\numberwithin{equation}{section}
\newcommand{\C}{\mathbb{C}}
\newcommand{\tr}{\operatorname{tr}}
\newcommand{\Tr}{\operatorname{Tr}}
\newcommand{\vers}{\mathop{\longrightarrow}} 
\newcommand{\1}{1\!\!{\sf I}}
\newcommand{\R}{\mathbb{R}}

 \newtheorem{theorem}{Theorem}[section]
 \newtheorem{proposition}{Proposition}[section]

\newtheorem{lemma}{Lemma}[section]
\newtheorem{corollary}{Corollary}[section]
\newtheorem{definition}{Definition}[section]
\begin{document}

\title{\normalsize{\uppercase{\bf \fontsize{10}{13} Non universality of fluctuations  of outliers for Hermitian  polynomials in a complex Wigner matrix and a spiked diagonal matrix}}}
\author{\uppercase{\fontsize{8}{10} Mireille Capitaine}\thanks{\it \fontsize{8}{10}\selectfont Institut de Math\'ematiques de Toulouse; UMR5219; Universit\'e de Toulouse; CNRS;
 UPS, 118 rte de Narbonne F-31062 Toulouse, FRANCE. 
E-mail: mireille.capitaine@math.univ-toulouse.fr }}

\date{}
\maketitle
\begin{abstract}{ \fontsize{8}{10}\selectfont 
We study the fluctuations associated to the a.s. convergence of the outliers established  
by Belinschi-Bercovici-Capitaine of an Hermitian polynomial in a complex Wigner matrix and a spiked deterministic real diagonal matrix. Thus, we extend the  non universality phenomenon established by Capitaine-Donati Martin-F\'eral  for additive deformations of complex Wigner matrices, to any Hermitian polynomial.  The result is described using
the operator-valued subordination functions of free probability theory.\\

\noindent {\it Key words:} Random matrices; Free probability;  Outliers; Fluctuations; Nonuniversality; Linearization; Operator-valued subordination.\\

\noindent Mathematics Subject Classification 2000: 15A18, 15B52,60B20, 60F05, 46L54}
\end{abstract}

\section{Introduction}

There is currently a quite precise knowledge of the asymptotic  spectral properties
 (i.e. when the dimension of the matrix tends to infinity) of a
 number of ``classical" random matrix models (Wigner matrices,
 Wishart matrices, invariant ensembles...). {This understanding covers both the so-called
 global regime (asymptotic behavior of the spectral measure) and the local regime
(asymptotic behavior of the extreme eigenvalues and  eigenvectors, spacings...).} We refer to the monographies \cite{AGZ, BS10, D, F, Me, PS}  for a thorough introduction to random matrix theory.
 
 Practical problems (in the theory of statistical learning, signal detection
etc.) naturally lead to wonder about the spectrum reaction
of a given random matrix after a deterministic perturbation. For example,
in the signal theory, the deterministic perturbation  is seen as the signal,
the perturbed matrix is perceived as a noise, and the question is to know whether the observation of the spectral properties of signal plus noise can give access to significant  parameters on the signal.
Theoretical results on these ``deformed" random models may allow  to establish
statistical tests on these parameters. {A typical illustration is the so-called BBP phenomenon
(after Baik, Ben Arous, P\'ech\'e \cite{BBP}) which put forward outliers
(eigenvalues that move away from the rest of the spectrum) and their Gaussian fluctuations
for spiked covariance matrices.}\\

 P\'ech\'e \cite{Peche} established Gaussian  fluctuations for any outlier of a low rank  additive  deformation of a   G.U.E. matrix. 
Fluctuations of outliers for additive finite rank deformations of non-Gaussian Wigner matrices  have been studied in \cite{CDF, CDF1, FP, RenfrewSos1, RenfrewSos2}.
It turns out  that the limiting distribution   depends on the localisation/delocalisation of the eigenvectors associated to the non-null eigenvalues of the perturbation. Note that in the G.U.E. case investigated by P\'ech\'e \cite{Peche},
the eigenvectors of the perturbation are irrelevant for the 
fluctuations, due to the unitary invariance in
Gaussian models. 
Let us illustrate this dependence on the eigenvectors of the perturbation in a very simple situation. Let $W_N=(W_{ij})_{1\leq i, j\leq N}$ be a $N\times N$ Hermitian  Wigner matrix  where $\{W_{ii}, \sqrt{2}\mathcal{R}W_{ij}, \sqrt{2}\mathcal{I}W_{ij}\}_{1\leq i< j}$ 
are independent identically distributed  random variables with law $\mu$, 
 $\mu$ is a symmetric distribution, with variance $1$,  and  satisfies a Poincar\'e inequality (see the Appendix). Note that when $\mu$ is Gaussian, $W_N$ is a G.U.E matrix. 
Consider two finite rank perturbations of rank 1, with one non-null eigenvalue $\theta >1$. The first one $A_N^{(1)}$ is a matrix with all entries equal to $\theta/N$ (delocalized eigenvector associated to $\theta$). The second one $A_N^{(2)}$ is a diagonal matrix (localized eigenvector associated to $\theta$). The limiting spectral distribution of each matrix $M_N^{(i)} = \frac{W_N}{\sqrt{N}} + A_N^{(i)}$ ($i=1,2$) is the semi-circular distribution $$d\mu_{sc}(t)=
\frac{1}{2\pi } \sqrt{4-t^2}{\1}_{[-2;2]}(t) dt.$$ Nevertheless the largest eigenvalue  $\lambda_1$ of each matrix $M_N^{(i)}$ ($i=1,2$)
separates from the bulk and converges towards $\rho_{\theta} := \theta + \frac{1}{\theta} (>2)$.
The fluctuations of   $\lambda_1$ 
around 
$\rho_{\theta}$
  are given as follows :
\begin{proposition}\label{flu} 
\begin{enumerate} 
\item Delocalized case \cite{FP}: The largest eigenvalue $\lambda_1(M^{(1)}_N)$ have Gaussian fluctuations,
\begin{equation} \label{delocalized}
\sqrt{N}  (\lambda_1(M_N^{(1)}) - \rho_\theta) \stackrel{{\cal D}}{\longrightarrow}  {\cal N}(0, 1 - 1/ \theta^2).
\end{equation}
\item Localized case 
\cite{CDF}:  The largest eigenvalue $\lambda_1(M^{(2)}_N)$ fluctuates as 
\begin{equation} \label{localized}
 \sqrt{N} (1 - \frac{1}{\theta^2}) (\lambda_1(M_N^{(2)}) - \rho_\theta) \stackrel{{\cal D}}{\longrightarrow} \mu \star {\cal N}(0, v_\theta).
 \end{equation}
where 
the variance $v_\theta$ of the Gaussian distribution depends on $\theta$ and the  fourth moment of $\mu$.
\end{enumerate}
\end{proposition} 

Hence, for localized eigenvectors of the perturbation, the limiting distribution depends on the distribution  of the entries of the Wigner matrix and thus, this uncovers a non universality phenomenon. This paper wants to extend such a non universality phenomenon for an additive deformation, to  general polynomials in a Wigner matrix and a diagonal deterministic matrix. Free probability is a main tool to achieve this purpose.\\

Free probability theory was introduced by Voiculescu around 1983 motivated by 
the isomorphism problem of von Neumann algebras of free groups. He developped a noncommutative probability theory, on a noncommutative
probability space, in which a new notion of freeness plays the
role of independence in classical probability. Around 1991, Voiculescu \cite{V}  threw a bridge connecting random matrix theory with free probability 
since he realized that the freeness property is also present for many classes of random matrices, in the
asymptotic regime when the size of the matrices tends to infinity. Since then,
several papers aimed
at developing the contribution of  free probability theory to the analysis
of the spectral properties of deformed ensembles and polynomials in random matrices. 
In particular, the main principle of subordination in free probability is emphasized
as a main tool in the understanding of the localization of the outliers and the corresponding eigenvectors 
of many matricial models. 
It was the purpose of
\cite{CD} to put forward an unified understanding  based on subordination in free probability
for studying the spectral properties of full rank deformations of classical Hermitian matrix models.
This investigation relies notably on \cite{CDFF, C, Cindia,  BBCF, Csem}.
This universal understanding culminates in \cite{BBC} dealing
with noncommutative polynomials in random Hermitian matrices; this investigation is achieved by an
even more general methodology based on a linearization procedure and operator-valued subordination
properties.\\

The aim of this paper is to study the fluctuations associated to the a.s. convergence of the outliers described in \cite{BBC} of an Hermitian polynomial in a Wigner matrix and a spiked deterministic Hermitian matrix ({\it spiked} means that the matrix has a fixed eigenvalue outside the  support of its limiting spectral measure). 
Capitaine and P\'ech\'e \cite{MirSand} established Gaussian  fluctuations for any outlier of a full rank additive  deformation of a   G.U.E. matrix using scalar-valued free probability theory. We generalize this result to any polynomial in a G.U.E. matrix and a deterministic Hermitian matrix which has a spike with multiplicity one, using operator-valued free probability theory. Moreover, considering any Hermitian  polynomial in a non-Gaussian Wigner  matrix and a deterministic real diagonal matrix which has a spike with multiplicity one, we establish that 
 the limiting distribution of outliers  is the classical convolution of a Gaussian distribution and the distribution  of the entries of the Wigner matrix; thus, this  extends the  non universality phenomenon \eqref{localized} established in  \cite{CDF} for additive deformations of Wigner matrices.   
The result is described in terms of operator-valued subordination functions 
related to a linearization of the noncommutative polynomial involved in the definition of our model. Therefore, we start by
 describing the necessary terminology and 
results concerning linearization procedure and free probability theory in Sections \ref{Linear} and \ref{freeproba}. In Section \ref{assumptions}, we present our matrix model and main results (Theorem \ref{principal} and Corollary \ref{gaussien}). Section \ref{prelim} gathers several preliminary results that will be used in Section \ref{Preuve} to prove Theorem \ref{principal}. An Appendix  recalls some basic facts on Poincar\'e inequalities and concentration phenomenon that are used in some proofs, as well as a basic development of the determinant of a  perturbation of a matrix.\\

To begin with, we introduce
some notations.
\begin{itemize}
\item $M_p(\C)$ is the set of $p\times p$ matrices  with complex entries,  $M_p^{sa}(\C)$ the subset of self-adjoint elements of $M_p(\C)$ and $I_p$ the identity matrix. In the following, we shall consider two sets of matrices with $p=m$ ($m$ fixed) and $p=N \text{~or~}N-1$ with $N\vers \infty$.
\item $\Tr_p$ denotes the trace and $\tr_p = \frac{1}{p} \Tr_p$ the normalized trace on $M_p(\C)$.
\item $|| . ||$ denotes the operator norm on $M_p(\C)$.
\item ${\rm id}_p$ denotes the identity operator from $M_p(\C)$ to $M_p(\C)$.
\item  $(E_{ij})_{i,j =1}^N$ (resp. $(E_{ij})_{i,j =1}^{N-1}$) denotes  the canonical basis of $M_N(\C)$ (resp.
$M_{N-1}(\C)$) whereas $(e_{pq})_{p,q=1}^m$ denotes  the canonical basis of $M_m(\C)$.
\end{itemize}
For any integer number $k$, we will say that a random  term in some $M_p(\C)$, depending on $N$ and  $w \in M_m(\mathbb{C})  $ such that $\Im  w$ is positive definite,  
 is  $O\left(\frac{1}{N^k}\right)$ if its  operator  norm  is smaller than
 $\frac{ Q\left(\Vert (\Im w)^{-1} \Vert\right) (\Vert w \Vert +1)^d}{{N}^k}$ for some deterministic  polynomial $Q$
 whose coefficients are nonnegative real numbers  and some integer number $d$  (which may depend on $m$).\\
 For a family of random terms $I_{i}$, $i \in \{1,\ldots,N\}^2$, we will set $I_{i}=O_{i}^{(u)} \left(\frac{1}{N^k}\right)$ if for each $i$, $I_{i}=O \left(\frac{1}{N^k}\right)$ and moreover one can find a bound of 
the norm of each $I_{i}$  as above involving a common polynomial $Q$ and a common $d$, that is not depending on $i$.\\

Throughout the paper, $K$, $C$ denote some positive constants 
that may depend on $m$ and  vary from line to line.

\section{A Linearization trick}\label{Linear}
A powerful tool to deal with noncommutative polynomials in random matrices or in operators is the so-called ``linearization trick'' that 
goes back to Haagerup and Thorbj{\o}rnsen \cite{HT05,HT06}  in the context of operator algebras and random
matrices (see \cite{MS}). 
 We use the procedure introduced in \cite[Proposition 3]{A}.\\
~~

Given a polynomial $P\in\mathbb C\langle X_1,\dots,X_k\rangle$, we call {\it linearization} of $P$ any $L_P\in M_m(\mathbb{C}) \otimes \mathbb{C} \langle X_1,\ldots, X_k \rangle$  such that 
 $$L_P := \begin{pmatrix} 0 & u\\v & Q \end{pmatrix} \in M_m(\mathbb{C}) \otimes \mathbb{C} \langle X_1,\ldots, X_k \rangle$$
where
\begin{enumerate}
\item $ m \in \mathbb{N}$,
\item $ Q \in M_{m-1}(\mathbb{C})\otimes \mathbb{C} \langle X_1,\ldots, X_k \rangle$ is invertible,
\item 
  u is a row vector and v is a column vector, both of size $m-1$ with
entries in $\mathbb{C} \langle X_1,\ldots, X_k \rangle$,
\item  the polynomial entries in $Q, u$ and $v$ all have degree $\leq 1$,\\
\item
$${P=-uQ^{-1}v}.$$
\end{enumerate}

It is shown in \cite{A} that, given a polynomial $P\in\mathbb C\langle X_1,\dots,X_k\rangle$,
there exist $m\in\mathbb N$ and a linearization  $L_P \in M_m(\mathbb{C}) \otimes \mathbb{C} \langle X_1,\ldots, X_k \rangle.$
The algebra of polynomials in noncommuting indeterminates
$X_1,\ldots,X_k$  becomes a $*$-algebra by anti-linear extension of
$(X_{i_1}X_{i_2}\cdots X_{i_l})^* = X_{i_l}\cdots X_{i_2}X_{i_1}$, $(i_1, \ldots, i_l)\in \{1,\ldots,k\}^l, l\in \mathbb{N}\setminus\{0\}$. It turns out that if $P$ is self-adjoint, $L_P$ can be chosen to be self-adjoint.\\
The well-known result about Schur complements (see  \cite[Chapter 10, Proposition 1]{MS}) yields then the following invertibility equivalence.  
\begin{lemma}\label{inversible}
Let  $P=P^*\in\mathbb{C}\langle X_1,\ldots,X_k\rangle$ and let
$L_P \in M_m(\mathbb{C})\otimes \mathbb{C}\langle X_1,\ldots,X_k\rangle$
be a linearization of P with the properties outlined above. 
 Let $y = (y_1,\ldots, y_k)$  be a k-tuple of self-adjoint operators in a unital ${\cal C}^*$-algebra ${\cal A}$. Then, for any $z\in \mathbb{C}$,  $ze_{11}\otimes 1_{\cal A}-L_P(y)$
is invertible if and only if $z 1_{\cal A}-P(y)$ is invertible and  we have \begin{equation}\label{coin}\left(ze_{11}\otimes 1_{{\cal A}} - L_P(y)\right)^{-1}=\begin{pmatrix}
\left(z1_{\cal A}-P(y)\right)^{-1} & \star\\\star & \star \end{pmatrix}.\end{equation}
\end{lemma}
Beyond the equivalence described above, we will use the following bound.
\begin{lemma}\label{alta lemma}\cite{BBC} 
Let $z_0\in 
\mathbb C$ be such that $z_0 1_{\cal A}-P(y)$ is invertible.  
There exist two polynomials $T_1$ and $T_2$  in $k$ commutative indeterminates, with nonnegative coefficients, depending only on $L_P$, such that \\

\noindent $\left\|(z_0e_{11}\otimes 1_{\cal A}-L_P(y))^{-1}\right\|$ $$ \leq T_1\left(\|y_1\|,\dots,\|y_k\|\right) \left\|(z_0 1_{\cal A}-P(y))^{-1}\right\|+ T_2\left(\|y_1\|,\dots,\|y_k\|\right).$$
Moreover, if the distance from $z_0$ to the spectrum of $P(y)$ is at least $\delta>0$,
 and for any $i\in\{1,\ldots,k\}$, $\|y_i\|\leq C$, for some positive real numbers $\delta$ and $C$, then 
there exists a constant $\varepsilon>0$, depending only on $L_P, \delta, C$ such that 
the distance from $0$ to the spectrum of  $(z_0e_{11}\otimes 1_{\cal A}-L_P(y))$ is at least $\varepsilon$.
\end{lemma}

\section{Free Probability Theory}\label{freeproba}
\subsection{\emph{Scalar-valued free probability theory}}\label{sca}
For the reader's convenience, we recall the following basic definitions from free probability theory. For a thorough introduction to free probability theory, we refer to \cite{VDN}.
\begin{itemize}
\item A ${\cal C}^*$-probability space, resp. a ${\cal W}^*$-probability space, is a pair $\left({\cal A}, \phi\right)$ consisting of a unital $ {\cal C}^*$-algebra ${\cal A}$, resp. of a unital von Neumann algebra,  and a state $\phi$ on ${\cal A}$ (i.e a linear map $\phi: {\cal A}\rightarrow \mathbb{C}$ such that $\phi(1_{\cal A})=1$ and $\phi(aa^*)\geq 0$ for all $a \in {\cal A}$), resp. a normal state. $\phi$ is a trace if it satisfies $\phi(ab)=\phi(ba)$ for every $(a,b)\in {\cal A}^2$. A trace is said to be faithful if $\phi(aa^*)>0$ whenever $a\neq 0$. 
An element of ${\cal A}$ is called a noncommutative random variable. 
\item The $\star$-noncommutative distribution of a family $a=(a_1,\ldots,a_k)$ of noncommutative random variables in a ${\cal C}^*$-probability space $\left({\cal A}, \phi\right)$ is defined as the linear functional $\mu_a:P\mapsto \phi(P(a,a^*))$ defined on the set of polynomials in $2k$ noncommutative indeterminates, where $(a,a^*)$ denotes the $2k$-tuple $(a_1,\ldots,a_k,a_1^*,\ldots,a_k^*)$.
For any self-adjoint element $a_1$ in  ${\cal A}$,  there exists a probability  measure $\nu_{a_1}$ on $\mathbb{R}$ such that,   for every polynomial P, we have
$$\mu_{a_1}(P)=\int P(t) \mathrm{d}\nu_{a_1}(t).$$
Then,  we identify $\mu_{a_1}$ and $\nu_{a_1}$. If $\phi$ is faithful then the  support of $\nu_{a_1}$ is the spectrum of $a_1$  and thus  $\|a_1\| = \sup\{|z|, z\in \rm{support} (\nu_{a_1})\}$. 
\item A family of elements $(a_i)_{i\in I}$ in a ${\cal C}^*$-probability space  $\left({\cal A}, \phi\right)$ is free if for all $k\in \mathbb{N}$ and all polynomials $p_1,\ldots,p_k$ in two noncommutative indeterminates, one has 
\begin{equation}\label{freeness}
\phi(p_1(a_{i_1},a_{i_1}^*)\cdots p_k (a_{i_k},a_{i_k}^*))=0
\end{equation}
whenever $i_1\neq i_2, i_2\neq i_3, \ldots, i_{n-1}\neq i_k$ and $\phi(p_l(a_{i_l},a_{i_l}^*))=0$ for $l=1,\ldots,k$.
\item A   noncommutative random variable $x$  in a ${\cal C}^*$-probability space  $\left({\cal A}, \phi\right)$ is a standard  semicircular variable if
 $x=x^*$  and for any $k\in \mathbb{N}$, $$\phi(x^k)= \int t^k d\mu_{sc}(t)$$
where $d\mu_{sc}(t)=
\frac{1}{2\pi} \sqrt{4-t^2}{\1}_{[-2;2]}(t) dt$ is the semicircular standard distribution.
\item Let $k$ be a nonnull integer number. Denote by ${\cal P}$ the set of polynomials in $2k $ noncommutative indeterminates.
A sequence of families of variables $ (a_n)_{n\geq 1}  =
(a_1(n),\ldots, a_k(n))_{n\geq 1}$ in ${\cal C}^* $-probability spaces 
$\left({\cal A}_n, \phi_n\right)$ converges, when n goes to infinity, respectively  in distribution if the map 
$P\in {\cal P} \mapsto
\phi_n(
P(a_n,a_n^*))$ converges pointwise
and strongly in distribution if moreover the map 
$P\in {\cal P} \mapsto  \Vert P(a_n,a_n^*) \Vert$  converges pointwise.
\begin{proposition}\cite[Proposition 2.1]{ColMal11}\label{spectres}
Let $x_n = (x_1(n),\ldots, x_k(n))$ and $x=(x_1,\ldots,x_k)$ be k-tuples of self-adjoint  variables in ${\cal C}^* $-probability spaces,
$\left({\cal A}_n, \phi_n\right)$ and 
$\left({\cal A}, \phi\right)$, with faithful states. Then, the following assertions are equivalent.
\begin{itemize}
\item $ x_n$ converges strongly in distribution to x,
\item for any self-adjoint variable $h_n= P(x_n)$, where P is a fixed polynomial, $\mu_{h_n}$ converges in weak-*
topology to $\mu_h$ where $ h = P(x)$. 
Moreover, the support of $\mu_{h_n}$ converges in Hausdorff distance to the support of $\mu_h$, that is: for
any $\epsilon> 0$, there exists $n_0$ such that for any $n \geq n_0$,
$$\rm{supp}(\mu_{h_n}) \subset \rm{supp}(\mu_h) +(-\epsilon,+\epsilon).$$
The symbol supp means the support of the measure.
\end{itemize}
\end{proposition}
\end{itemize}

Additive  free convolution arises as natural analogue of classical convolution
in the context of free probability theory. 
For two 
Borel probability measures $\mu$ and $\nu$ on the real line, one defines the free 
additive convolution $\mu\boxplus\nu$ as the distribution of $a+b$, where $a$ and $b$
are free self-adjoint random variables with distributions $\mu$ and
$\nu$, respectively.  We refer  to  \cite{BercoviciVoiculescu, Maassen, Voiculescu86} for the
definitions and main properties of free convolutions. Let us briefly recall 
 the fundamental analytic subordination
properties \cite{ Biane98, voic-fish1, V2000} of this convolution.
The analytic subordination phenomenon for free additive convolution was first noted by Voiculescu in \cite{voic-fish1} for free 
additive convolution of compactly supported probability measures.  Biane \cite{Biane98}
extended the result to free additive convolutions of arbitrary probability measures on $\mathbb R$.
A new proof was
given later, using a fixed point theorem for analytic self-maps of the upper
half-plane \cite{BelBer07}. Note that such a subordination property  allows to give
a new definition of free additive convolution \cite{CG08a}. 
 Let us define  the reciprocal Cauchy-Stieltjes
transform $F_{\mu}(z)={1}/{g_{\mu}(z)}$, which is an analytic
self-map of the upper half-plane, where $g_\mu: z\in \mathbb{C} \setminus \mathbb{R} \mapsto  \int \frac{1}{z-t}d\mu(t)$.  Given Borel probability measures $\mu$ and $\nu$ on 
$\mathbb R$, there exist a  unique pair of  analytic functions $\omega_1,\omega_2\colon
\mathbb C^+\to\mathbb C^+$ such that 
\begin{equation}\label{EqSubord+}
F_\mu(\omega_1(z))=F_\nu(\omega_2(z))=F_{\mu\boxplus
\nu}(z),\quad z\in\mathbb C^+.
\end{equation}
Moreover $\lim_{y\to+\infty}\omega_j(iy)/iy=1$, $j=1,2$ and $$\omega_1(z)+\omega_2(z)-z=F_{\mu\boxplus
\nu}(z),\quad z\in\mathbb C^+.$$
\noindent  In particular (see \cite{BelBer07}), for any $z\in\mathbb C^+\cup\mathbb R$ so that 
$\omega_1$ is analytic at $z$, $\omega_1(z)$ is the attracting fixed point of the self-map
of $\mathbb C^+$ defined by 
$$
 w\mapsto F_\nu(F_\mu(w)-w+z)-(F_\mu(w)-w).
$$
A similar statement, with $\mu,\nu$ interchanged, holds for $\omega_2$.

\noindent In particular, according to (\ref{EqSubord+}), we have for any $z \in \mathbb C^+,$
\begin{equation}\label{subord1}
g_{\mu\boxplus\nu}(z)=g_\mu(\omega_1(z))=g_\nu(\omega_2(z)).
\end{equation}

\subsection{\emph{Operator-valued free probability theory}}\label{opfree}
There exists an extension, operator-valued
free probability theory, which still shares the basic properties of free probability but
is much more powerful because of its wider domain of applicability.
The concept of freeness with amalgamation and some of the relevant 
analytic transforms 
were introduced by Voiculescu
in \cite{V1995}.
\begin{definition}
Let $ {\cal M}$ be an algebra and  ${\cal B}\subset {\cal M}$ be a unital subalgebra. A linear map $E: {\cal M} \rightarrow {\cal B}$ is a { conditional expectation} if $E(b)=b $ for all $b \in {\cal B}$ and $E(b_1 a b_2) =b_1 E(a) b_2$ for all $a \in {\cal M}$  and $b_1,b_2$ in ${\cal B}$. Then
$\left({\cal M},E\right)$ is called a
{  ${\cal B}$-valued probability space}. 
If in addition ${\cal M}$ is a ${\cal C}^*$-algebra, ${\cal B}$ is a ${\cal C}^*$-subalgebra of ${\cal M}$ and $E$ is completely positive, then we have { a ${\cal B}$-valued ${\cal C}^*$-probability space}.
\end{definition}

\noindent Example: Let $({\cal A}, \phi)$ be a noncommutative probability space.  Define
$$M_2({\cal A}):= \left\{ \begin{pmatrix} a&b \\c &d \end{pmatrix}, a,b,c,d \in {\cal A}\right\},~~E:= {\rm id}_2 \otimes \phi \mbox{~~that is}$$

$$E \left[ \begin{pmatrix} a&b \\c &d \end{pmatrix} \right] =  \begin{pmatrix} \phi(a)&\phi(b) \\\phi(c) &\phi(d) \end{pmatrix}.$$
$\left( M_2({\cal A}),E \right)$ is an $M_2(\mathbb{C})$-valued probability space ($\mathbb{\mathbb C} \approx \mathbb{C} 1_{\cal A}$).\\

As in scalar-valued free probability, one defines \cite{V1995} {\em freeness with amalgamation}
over ${\cal B}$ via an algebraic relation similar to {freeness}, but involving $E$
 and noncommutative polynomials with coefficients in ${\cal B}$.

\begin{definition}
Let $\left( {\cal M}, E: {\cal M} \rightarrow {\cal B} \right)$ be  an operator-valued probability space.\\
{ The ${\cal B}$-valued distribution} of a noncommutative random variable $a \in {\cal M}$ is given by all ${\cal B}$-valued moments $E(ab_1ab_2\cdots b_{n-1} a) \in {\cal B}$, $n \in \mathbb{N}, b_0, \ldots, b_{n-1} \in {\cal B}$.\\
 Let $(A_i)_{i\in I}$ be a family  of subalgebras with ${\cal B}\subset A_i$ for all $i\in I$. The subalgebras $(A_i)_{i\in I}$ are
free with respect to E or free with amalgamation over ${\cal B}$ if
$E(a_1\cdots a_n)=0$ whenever $a_j \in A_{i_j}$, $i_j \in I$, $E(a_j)=0$, for all $j$ and $i_1 \neq i_2 \neq \cdots \neq i_n$.

 Random variables in
${\cal M}$  or subsets of ${\cal M}$ are free with amalgamation over ${\cal B}$ if the algebras generated by
${\cal B}$ and the variables or the algebras generated by ${\cal B}$ and the subsets, respectively, are
so.

\end{definition}
 A centred 
${\cal B}$-valued semicircular random variable $s$ is  uniquely determined by its variance $\eta\colon b\mapsto E(sbs)$;
a characterization in terms of moments
and cumulants via $\eta$ is provided by Speicher in \cite{SMem}.  \\

The previous results of free subordination property in the scalar case are approached from an abstract coalgebra point of view by Voiculescu in \cite{V2000} and this approach extends the results to the ${\cal B}$-valued case. In \cite{BMS}, Belinschi, Mai and Speicher  develop an analytic theory. 
  In order to describe operator-valued subordination property, we need some notation. 
If $\mathcal A$ is a unital ${\cal C}^*$-algebra
and $b\in\mathcal A$, we denote by $\Re b=(b+b^*)/2$ and 
$\Im b=(b-b^*)/2i$ the real and imaginary
parts of $b$, so  $b=\Re b+i\Im b$. For a self-adjoint operator 
$b\in\mathcal A$, we write $b\ge 0$
if the spectrum  of $b$ is contained in $[0,+\infty)$
and  $b>0$ if the spectrum  of $b$ is contained in $
(0,+\infty)$. The operator upper
half-plane of $\mathcal A$ is the set 
$\mathbb H^+(\mathcal A)=\{b\in\mathcal A\colon\Im b>0\}$.

\begin{proposition}\label{subop}\cite{V2000},\cite{BMS}(see Theorem 5 p 259 \cite{MS}) 
Let $\left(\mathcal M,{E:{\cal M} \rightarrow {\cal B}} \right)$ be an operator-valued ${\cal C}^*$-probability space.
 Let {$x_1,x_2\in\mathcal M$} be
self-adjoint variables which are {free with amalgamation over ${\cal B}$}. \\
There exist a unique pair of Fr\'echet  analytic maps {$\omega_1, \omega_2\colon\mathbb H^+({\cal B})\to\mathbb H^+({\cal B})$}
such that, for all $b\in\mathbb H^+({\cal B})$,
\begin{itemize}
\item  \begin{equation} \label{imw}\Im \omega_{j}(b)\ge\Im  b, \; j=1,2; \end{equation}
 \item 
$$
E\left[(b-(x_1+x_2))^{-1}\right]=E\left[(\omega_1(b)-x_1)^{-1}\right]=E\left[(\omega_2(b)-x_2)^{-1}\right],$$
\item  \begin{eqnarray*}\left\{E\left[(\omega_1(b)-x_1)^{-1}\right]\right\}^{-1}+b &=&
\left\{E\left[(\omega_2(b)-x_2)^{-1}\right]\right\}^{-1}+b\\&=&\omega_1(b)+\omega_2(b).
\end{eqnarray*}
\end{itemize}
Moreover, if $b \in \mathbb H^+({\cal B})$, then $ \omega_1(b)$ is the unique fixed point of the map 
$$f_b:\mathbb H^+({\cal B}) \rightarrow \mathbb H^+({\cal B}),~~f_b(w)=h_{x_2}(h_{x_1}(w)+b)+b$$
$$\text{where~~}h_{x_i}(b)=E\left[(b-x_i)^{-1}\right]^{-1}-b$$
 $$\text{and~~}\omega_1(b)= \lim_{k\rightarrow +\infty} f_b^{\circ k}(w),\text{~~for any~~}  w\in\mathbb H^+({\cal B}).$$
\end{proposition}
The following result from \cite{NSS} explains why the   particular case  ${\cal B}=M_m(\mathbb C)$, $\mathcal M=M_m({\cal A})$, $E={\rm id}_m\otimes \phi$, where $({\cal A}, \phi)$ is a noncommutative probability space,
 is relevant in our work using linearizations of polynomials.

\begin{proposition}\label{fr}
Let $(\mathcal{A},\phi)$ be a ${\cal C}^*$-probability space, let $m$ be a 
positive integer, 
 and let $x_1,\ldots,x_n \in\mathcal A$ be freely independent. Then the map 
 ${\rm id}_m\otimes\phi \colon M_m(\mathcal A)\to M_m (\mathbb C)$ 
 is a unit preserving conditional expectation, and
 $\alpha_1\otimes x_1, \ldots, \alpha_n \otimes x_n$  are free over $M_m (\mathbb C)$ for any
$\alpha_i\in M_m(\mathbb C)$.
\end{proposition} 
Now, if $x$ is a standard scalar-valued semicircular
centred noncommutative  random variable which is free from a self-adjoint variable $a$ in some ${\cal W}^*$-probability space $({\cal A}, \phi)$, then,  for any  Hermitian matrices $\alpha$, $\beta$ in $M_m(\C)$,  $\alpha \otimes x$ is a $ M_m(\mathbb C)$-valued semicircular of variance $\eta\colon b\mapsto \alpha b\alpha$ which is free over $ M_m(\mathbb C)$ from $\beta \otimes a$ and  the subordination
function  has a more explicit form  (see \cite[Chapter 9]{MS} and the end of the proof of Theorem 8.3 in \cite{ABFN}): for $b \in \mathbb{H}^+(M_m(\C))$,
$$\left({\rm id}_m \otimes \phi\right)  \left[\left(b\otimes 1_{\cal A} - \alpha \otimes x-\beta \otimes a \right)^{-1}\right]=\left({\rm id}_m \otimes \phi\right) \left[\left(\omega_m(b)\otimes 1_{\cal A} -\beta\otimes a\right)^{-1}\right],$$
where
 \begin{equation}\label{defomega}
\omega_m(b)=b -\alpha \left({\rm id}_m \otimes \phi\right) \left[\left(b\otimes 1_{\cal A} - \alpha \otimes x-\beta \otimes a \right)^{-1}\right]\alpha.
\end{equation} 
Denote by  ${\cal N}$ the unital von Neumann algebra  generated by $M_m(\C)$ and $\beta\otimes a$ and by  $E_{\cal N}$  the unique trace preserving conditional expectation of $M_m({\cal A})$ onto ${\cal N}$.  Actually the following strengthened result \cite[Theorem 3.8]{V2000} holds:
\begin{equation}\label{inversi} E_{\cal N} \left[\left(b\otimes 1_{\cal A} - \alpha \otimes x-\beta \otimes a \right)^{-1}\right]=\left(\omega_m(b)\otimes 1_{\cal A} -\beta\otimes a\right)^{-1}.\end{equation}
\section{Assumptions and main results}\label{assumptions}
{\bf Assumptions on the Wigner matrix.}\\
We consider a $N\times N$ Hermitian  Wigner matrix $W_N=(W_{ij})_{1\leq i, j\leq N}$ such that the random variables  $\{W_{ii}, \sqrt{2}\mathcal{R}W_{ij}, \sqrt{2}\mathcal{I}W_{ij}\}_{1\leq i< j}$ 
are independent identically distributed   with law $\mu$, 
 $\mu$ is a centered distribution, with variance 1,  and  satisfies a Poincar\'e inequality (see the Appendix).
We set  $$W_N=\begin{pmatrix} W_{11} & Y^*\\ Y & W_{N-1} \end{pmatrix},$$ 
where $Y^*=(W_{12}, ~ \cdots ~, W_{1N}) $ and $W_{N-1} \in M_{N-1}(\mathbb{C})$.\\

\noindent {\bf Assumptions on the deterministic matrix.}\\
We consider a deterministic real diagonal matrix  $A_N$: $$A_N = {\rm diag}(\theta,A_{N-1})$$ where 
$\theta \in \mathbb R$ is independent of $N$ and $A_{N-1}$ is a $N-1 \times N-1$ deterministic diagonal
matrix such that for any $i=1,\ldots,N-1$, $(A_{N-1})_{ii}=d_i(N)$. 
We assume that  $A_{N-1} \in (M_{N-1}(\mathbb{C}), \frac{1}{N-1} \Tr)$  converges strongly in distribution  towards a noncommutative self-adjoint random variable $a$ in some ${\cal W}^*$-probability space $({\cal A}, \phi)$, with $\phi$ faithful (see Section  \ref{sca} for the definition of strong convergence). Note that this implies that 
\begin{equation}\label{sup}\sup_{N} \Vert A_{N-1}\Vert <+\infty,\end{equation}
and, by Proposition \ref{spectres}, that,  for all large $N$,  all the eigenvalues of $A_{N-1}$ are in any small neighborhood of the spectrum of $a$.  We assume that $\theta$  is such that $\theta \notin \text{supp}(\mu_a) =\text{spect}(a)$. Note that the previous assumptions yield that $A_{N} \in (M_{N}(\mathbb{C}), \frac{1}{N} \Tr)$  converges  in distribution  towards the noncommutative  random variable $a$ and that,  for $N$ large enough,  $\theta$ is an eigenvalue of multiplicity 1 of $A_N$.\\

\noindent {\bf Matrix model.}\\
Fix a self-adjoint polynomial $P \in \mathbb{C}<X_1,X_2>$. The matrix model we are interested in is 
$$M_N=P\left( \frac{W_N}{\sqrt{N}},A_N\right).$$
Denote by $\lambda_i(M_N), i=1,\ldots,N$, its eigenvalues and by $$\mu_{M_N}=\frac{1}{N}\sum_{i=1}^N \lambda_i(M_N)$$ its empirical spectral measure. According to (2.10) in \cite{BC}
and \cite[Theorem 5.4.5]{AGZ}, we have
$$
\lim_{N\to\infty}\mu_{M_N}=\mu_{P(x,a)}
$$
almost surely in the weak${}^*$ topology, 
 where $x$ is a standard semicircular noncommutative random variable in $({\cal A}, \phi)$ (i.e $d\mu_x =\frac{1}{2 \pi } \sqrt{4- x^2} 
\, 1 \hspace{-.20cm}1_{[-2 , 2 ]}(x)$),
$a$ and $x$ are freely independent,
  and $\mu_{P(x,a)}$ denotes the distribution of $P(x,a)$.\\
	
The set of outliers of $M_N$
 is calculated in \cite{BBC}
 from the spike 
$\theta$ of $A_N$  using linearization and Voiculescu's
matrix subordination function \cite{V2000} as follows. 
Choose a linearization $L_P$  of $P$ where $ L_P=
\gamma \otimes 1 + 
\alpha\otimes X_1 + \beta \otimes X_2
$, $\alpha,
\beta, \gamma$
are self-adjoint  matrices in ${ M}_m(\mathbb{C})$, and
let $\omega_m$  be the subordination function associated to the semicircular operator-valued
random variable $\alpha\otimes x$ with respect to $\beta\otimes a$, as defined by \eqref{defomega}. {According to Lemma \ref{inversible}, $\omega_m$ extends as an analytic map $z \mapsto \omega_m(ze_{11} -\gamma)$  to $\mathbb{C}\setminus \mathrm{supp}(\mu_{P(x,a)})$.}
For any {$\rho \not\in {\mathrm{supp}}(\mu_{P(x,a)})$}, define  
$m(\rho)$ as the multiplicity of $\rho$ as a zero of $ \det(\omega_{m}
(\rho e_{11}-\gamma) - \theta \beta).$
\cite{BBC} establishes the following.
\begin{proposition}\label{BBCresultat}\cite{BBC} There exists $\delta_0>0$ such that,  for any    $0<\delta \leq\delta_0$,
a.s for large $N$,
there are exactly 
 $m(\rho)$ eigenvalues  of $
P\left(\frac{W_N}{\sqrt{N}},A_N\right)$ in  $]\rho-\delta; \rho+\delta[$, counting multiplicity.
\end{proposition}
\noindent {\bf Assumptions on $\rho$.}\\
In this paper, we assume that
   there exists some real number $$\rho \not\in {\mathrm{supp}}(\mu_{P(x,a)})=\mathrm{spect}(P(x,a))$$ such that $\rho$ is a zero with multiplicity one of \begin{equation}\label{zero}\det(\omega_{m}
(\rho e_{11}-\gamma) - \theta\beta)=0,\end{equation} 
that is such that \begin{equation}\label{egalun}m(\rho)=1. \end{equation}

\noindent {\bf Assumptions on $\epsilon$.}\\
{ Throughout the paper $\epsilon>0$ is fixed  such that \begin{equation}\label{distderho} d(\rho, \text{spect}(P(x,a)))>\epsilon \end{equation} and $$\det(\omega_{m}
(y e_{11}-\gamma) - \theta\beta)\neq 0, \; \text{ for any }\; y\in ]\rho-\epsilon; \rho +\epsilon[ \setminus \{\rho\}.$$}\\

\noindent {\bf Main result.}\\
We first introduce events and objects needed to state our main result.

\noindent  By strong asymptotic freeness of \cite{BBC} and Proposition \ref{spectres}, almost surely for all large $N$, 
 $\text{spect}\left(P\left(\frac{W_{N-1}}{\sqrt{N}},A_{N-1}\right)\right)\subset \{y \in \mathbb{R}; d(y, \text{spect}(P(x,a))\leq \epsilon/2\}.$
Thus,   \begin{equation}\label{hyprho} \text{~almost surely  for all large } N,\; d\left(\rho, \text{spect}\left(P\left(\frac{W_{N-1}}{\sqrt{N}},A_{N-1}\right)\right)\right)> \epsilon/2. \end{equation}
Define the event \begin{equation}\label{defOmegaN}\tilde \Omega_{N-1}= \left\{ d\left(\rho, \text{spect}\left(P\left(\frac{W_{N-1}}{\sqrt{N}},A_{N-1}\right)\right)\right)> \epsilon/2; \left\|\frac{W_{N-1}}{\sqrt{N}}\right\| \leq 3\right\},\end{equation}
Note that according to Lemma \ref{alta lemma}, there exists $C_\epsilon>0$ such that \begin{equation}\label{distanceinfini}d(0, \text{spect}((\rho e_{11} -\gamma)\otimes 1_{\cal A} -\alpha \otimes  x-\beta \otimes a))>C_\epsilon\end{equation} and on $\tilde \Omega_{N-1}$
\begin{equation}\label{distance}d\left(0, \text{spect}\left((\rho e_{11} -\gamma)\otimes I_{N-1} -\alpha \otimes  \frac{W_{N-1}}{\sqrt{N}}-\beta \otimes A_{N-1}\right)\right)>C_\epsilon.\end{equation} Let $\delta_0$ be as defined in Proposition \ref{BBCresultat}.
Set \begin{equation}\label{deftau}\tau=\min(\delta_0, \epsilon/4, C_\epsilon/4).\end{equation}
Define the event \begin{equation}\label{defoublie}\Omega_{N}=\tilde \Omega_{N-1}\cap  \left\{ \text{card}(\text{spect}(M_N) \cap ]\rho-\tau; \rho+\tau[)=1\right\}.\end{equation}
It readily follows from Proposition \ref{BBCresultat}, \eqref{hyprho} and Bai-Yin's theorem (see Theorem 5.1 in \cite{BS10}) that 
$$\lim_{N\rightarrow +\infty} \1_{\Omega_N}=1, ~\text{a.s.}$$ and then
$$\mathbb{P}(\Omega_N)\rightarrow_{N\rightarrow +\infty} 1.$$
Now, define \begin{equation}\label{deflambda}\lambda(N,\rho) =\left\{ \begin{array}{ll}\rho \text{~if~} \text{spect}(M_N) \cap ]\rho-\tau;\rho+\tau[=\emptyset\\
 \max \{ \text{spect}(M_N) \cap ]\rho-\tau;\rho+\tau[\} \text{~else}. \end{array} \right.\end{equation}
 On $\Omega_N$, $\lambda(N,\rho)$ is the unique eigenvalue of $M_N$ which is located in $]\rho-\tau; \rho+\tau[$. In this paper, we study the fluctuations of $\lambda(N,\rho)$. 
 Note that Proposition \ref{BBCresultat} readily implies that \begin{equation} \label{convlambda}\lambda(N,\rho) \rightarrow_{N\rightarrow +\infty} \rho \text{~a.s.}. \end{equation}
 Let $a_{N-1}$ be a self-adjoint noncommutative random variable in $({\cal A}, \phi)$ whose distribution is $\mu_{A_{N-1}}$ (meaning that 
$\forall k\in \mathbb{N}$, $\frac{1}{N-1} \Tr (A_{N-1}^k)=\phi((a_{N-1}^k)$) and which is free with the semicircular variable $x$. Since  $A_{N-1}$ (and thus $a_{N-1}$) converges strongly to $a$, we have, for all large $N$,$$
\text{spect}(P( x, a_{N-1}))  \subset
\text{spect}(P( x, a)) + ]-{\epsilon}/4,\epsilon/4[,
$$
and thus, using \eqref{distderho}, for any $z \in B(\rho,\tau):=\left\{z \in \mathbb{C}, \vert z-\rho \vert <\tau\right\}$, \begin{equation}\label{normederesdepenz} \left\|\left(z 1_{\cal A}-P( x, a_{N-1})\right)^{-1} \right\|\leq 2/\epsilon .\end{equation}
Define for any $\kappa \in\mathbb H^+( M_m(\mathbb C))$
{\begin{equation}\label{defomegam}\omega_m^{(N)}(\kappa)=\kappa -\alpha \left({\rm id}_m \otimes \phi\right) \left[ (\kappa \otimes 1_{\cal A} -\alpha \otimes x -\beta\otimes a_{N-1})^{-1}\right]\alpha.\end{equation}}
$\omega_m^{(N)}$ is the subordination function associated to the semicircular operator-valued
random variable $\alpha\otimes x$ with respect to $\beta\otimes a_{N-1}$. {According to Lemma \ref{inversible}, $\omega_m^{(N)}$ extends as an analytic map $z \mapsto \omega_m^{(N)}(ze_{11} -\gamma)$  to $\mathbb{C}\setminus \mathrm{supp}(\mu_{P(x,a_{N-1})})$.}
Using \eqref{normederesdepenz} and Lemma \ref{alta lemma}, it is straightforward to see that $\left(z\mapsto \det(\omega_{m}^{(N)}
(z e_{11}-\gamma) - \theta\beta)\right)_{N\geq 1}$ is a bounded sequence in the set of analytic functions endowed with the uniform convergence on compact subsets of  $B(\rho,\tau)$; therefore, using moreover  \eqref{convomegaz} and Vitali's theorem, by Hurwitz's theorem, \eqref{zero} yields  that  for any $0< \tau^{'}<\tau$, for all large $N$, there exists one and only one 
$\rho_N$  in $B(\rho,\tau^{'})$,  such that \begin{equation}\label{rhoN}\det(\omega_{m}^{(N)}
(\rho_N e_{11}-\gamma) - \theta\beta)=0,\end{equation} and we have \begin{equation} \label{convrho} \rho_N \rightarrow_{N\rightarrow +\infty} \rho. \end{equation} Moreover, necessarily $\rho_N$ is real since \eqref{rhoN} implies that 
$\det(\omega_{m}^{(N)}
(\overline{\rho_N} e_{11}-\gamma) - \theta\beta)=0.$\\
Here is our main result. (For a matrix $X$, $com(X)$ denotes the comatrix of $X$.)

\begin{theorem}\label{principal}
Define \begin{equation}\label{Cm}{\bf C}_m=^t com(\omega_m (\rho e_{11} -\gamma) -\beta \theta),\end{equation}
\begin{equation}\label{rappel} R_\infty(\rho e_{11}-\gamma)=((\rho e_{11}-\gamma)\otimes 1_{\cal A}-\alpha \otimes x -\beta \otimes a)^{-1},\end{equation}
\begin{equation}\label{crhoun}C_\rho^{(1)}= \Tr_m\left({\bf C}_m\left[ e_{11}+ \alpha \left({\rm id}_m \otimes \phi\right) (R_\infty(\rho e_{11}-\gamma)
\left( e_{11}\otimes 1_{\cal A}\right) R_\infty( \rho e_{11}-\gamma))\alpha\right]\right),\end{equation}
\begin{equation}\label{crhodeux}C_\rho^{(2)}=\Tr_m \left[ {\bf C}_m \alpha\right],\end{equation}
 $$v_\rho=\left(\mathbb{E}\left(\vert W_{21}\vert^4\right)-2\right) \int \left[
  \Tr_m \left(\alpha {\bf C}_m  \alpha \left( \omega_m(\rho e_{11} -\gamma) -t\beta\right)^{-1} \right)\right]^2d\mu_a(t)
  $$ \begin{equation}\label{vrho} +\phi \left(\left[  \left(\Tr_m \otimes {\rm id}_{\cal A}\right)\left\{ R_\infty(\rho e_{11}-\gamma) (\alpha {\bf C}_m \alpha) \otimes 1_{\cal A}\right\} \right]^2\right) ,\end{equation}
  with $\omega_m$ defined by \eqref{defomega}.
	
\noindent 
 $C^{(1)}_\rho \sqrt{N} ( \lambda(N,\rho)- \rho_N) $  converges in distribution to the classical convolution of the distribution of $C_\rho^{(2)} W_{11}$ and  a Gaussian distribution  with mean 0 and variance $v_\rho$.
\end{theorem}
Using the unitarily invariance of the distribution of a G.U.E. matrix, we can readily deduce the following result. {\begin{corollary}\label{gaussien} Assume that  $W_N$ is a G.U.E. matrix. Let  $A_N$ be  a deterministic Hermitian matrix  such that its spectral measure $\mu_{A_N}$ weakly converges towards a compactly supported measure $\mu_a$, $\theta \notin  \mathrm{supp}(\mu_a)$ is a spiked eigenvalue  of $A_N$ with multiplicity one whereas the other eigenvalues of $A_N$ converge uniformly to the compact support of $\mu_a$. Then, under the assumptions  \eqref{zero} and \eqref{egalun},
\noindent $C^{(1)}_\rho \sqrt{N} ( \lambda(N,\rho)- \rho_N) $  converges in distribution to   a Gaussian distribution  with mean 0 and variance $$\tilde v_\rho = (C_\rho^{(2)})^2 +\phi \left(\left[  \left(\Tr_m \otimes {\rm id}_{\cal A}\right)\left\{ R_\infty(\rho e_{11}-\gamma) (\alpha {\bf C}_m \alpha) \otimes 1_{\cal A}\right\} \right]^2\right),$$
where $ \lambda(N,\rho)$, $\rho_N$, $ C^{(1)}_\rho$,  $C^{(2)}_\rho$, ${\bf C}_m$ and $R_\infty(\rho e_{11}-\gamma)$ are defined by   \eqref{deflambda}, \eqref{rhoN}, \eqref{crhoun}, \eqref{crhodeux}, \eqref{Cm} and \eqref{rappel} respectively.\end{corollary}}
\noindent {\bf Example}\\
As an illustration, consider the random matrix 
$$
M_N=A_N \frac{W_N}{\sqrt{N}} +\frac{W_N}{\sqrt{N}}A_N +\frac{W_N^2}{N},
$$
where $W_N$ is a   Wigner matrix of size $N$  such that $d\mu(x)= \frac{1}{2\sqrt{3}} \1_{[-\sqrt{3}; \sqrt{3}]}(x) dx$ and 
$$
A_N=\text{Diag}(\theta, 0,\ldots,0),\quad \theta\in\mathbb{R}\setminus\{0\}. 
$$
According to \cite[(4.6.6)]{BLG}, $\mu$ satisfies a Poincar\'e inequality.
In this case, $A_N$ has rank one,  and thus $a=0$.  It follows that the limit spectral measure $\Pi$ of $M_N$ is the same as the limit spectral measure of $W_N^2/N$.  Thus, $\Pi$ is the Marchenko-Pastur distribution  with parameter 1:
$$ d\Pi(t) =\frac{\sqrt{
\left(4-t\right)t}}{2\pi  t}1_{(0,4)}(t)dt.$$
The polynomial $P$ is $P(X_1,X_2) = X_2X_1+X_1X_2 + X_1^2$, $a=0$ and $x$ is the standard semi-circular distribution. An economical linearization of $P$ is provided by  
$L=\gamma \otimes 1 + \alpha \otimes X_1+\beta\otimes X_2$,  where
$$
 \quad \gamma = \begin{bmatrix} 0 & 0&0\\0& 0 & -1 \\0&-1&0\end{bmatrix},  \quad\alpha=
 \begin{bmatrix} 0 & 1&\frac{1}{2}\\1& 0 & 0 \\ \frac{1}{2}&0&0\end{bmatrix},\quad \beta = \begin{bmatrix} 0 & 0&1\\0& 0 & 0 \\ 1&0&0\end{bmatrix}.
$$ Thus, here $m=3$.
Denote by $$G_\Pi(z)=\int_0^4\frac1{z-t}\,d\Pi(t)=\frac{z-\sqrt{z^2-4z}}{2z},\quad z\in\mathbb C\setminus{[0,4]},$$
the Cauchy transform of the measure $\Pi$. This function satisfies the quadratic equation $zG_\Pi(z)^2-zG_\Pi(z)+1=0$. Suppose now that $t\notin [0,4]$.  Denoting by $E=\rm{id}_3 \otimes \phi:M_3(\mathcal A)\to M_3(\mathbb C)$ the usual expectation, since $a=0$, the function $\omega_3$ is computed as follows: 
$$
\omega_3(te_{11}-\gamma)=
E((te_{11}-\gamma-\alpha\otimes x)^{-1})^{-1},
\quad t\in\mathbb R\setminus[0,4].
$$
The inverse of $te_{11}-\gamma-\alpha\otimes x$ is then calculated explicitly and application of the expected value to its entries yields 
$$\omega_3(te_{11}-\gamma)= 
\begin{bmatrix} 
\frac{1}{G_\Pi (t)} &0&0 \\
0 &\frac{1}{tG_\Pi (t)}-1 & \frac{1}{2t G_\Pi (t)} +\frac{1}{2}\\
0&\frac{1}{2tG_\Pi (t)}+\frac{1}{2}&\frac{1}{4tG_\Pi (t)} -\frac{1}{4} 
\end{bmatrix}.$$ 
The equation $\det[\beta\theta-\omega(te_{11}-\gamma)]=0$ is easily seen to reduce to  
\begin{equation} 
\theta^2 G_\Pi (t)^2 -(1-G_\Pi (t))=0.\label{equation}
\end{equation}
Thus, the matrix $M_N$ exhibits one (negative) outlier when $0<|\theta|\le\sqrt 2$
$$\rho_\theta^{-}=\frac{2\theta^4}{-(3\theta^2+1)-\sqrt{4\theta^2+1} (\theta^2 +1)},$$  and two outliers (one negative and one $>4$)  when $|\theta|>\sqrt 2$:
$$\rho_\theta^{\pm}=\frac{2\theta^4}{-(3\theta^2+1)\pm\sqrt{4\theta^2+1} (\theta^2 +1)};$$ 
note that $$g_{\rho_\theta^{\pm}}= G_\Pi (\rho_\theta^{\pm})=\frac{1}{2} + \frac{-(\theta^2+1) \pm \sqrt{4\theta^2+1}}{2\theta^2}.$$ 
Let  $\rho$ be any of the two solutions $\rho_\theta^+$ and $\rho_\theta^-$ and set $$ g_\rho= G_\Pi (\rho).$$
Note that since here $a_{N-1}=a=0$, we have $\rho_N=\rho$.
After computations 
$${\bf C}_3= \begin{pmatrix} g_\rho-1 & \frac{\theta}{2} (g_\rho-2) & -g_\rho\theta\\ \frac{\theta}{2} (g_\rho-2) & -(\frac{1}{4} +\theta^2) & -\frac{1}{g_\rho} +\frac{1}{2}\\-g_\rho\theta & -\frac{1}{g_\rho} +\frac{1}{2}& -1\end{pmatrix},$$
$$ R_\infty (\rho e_{11}-\gamma)=\begin{pmatrix} (\rho -x^2)^{-1} & \frac{1}{2} x (\rho -x^2)^{-1}& x(\rho -x^2)^{-1}\\
\frac{1}{2} x (\rho -x^2)^{-1}& \frac{1}{4} x^2 (\rho -x^2)^{-1}& 1+ \frac{1}{2} x^2 (\rho -x^2)^{-1}\\
 x (\rho -x^2)^{-1}&1+ \frac{1}{2} x^2 (\rho -x^2)^{-1}& x^2 (\rho -x^2)^{-1}\end{pmatrix},$$
and then $$ C_\rho^{(2)}=-2\theta,$$
$$C_\rho^{(1)}=-\theta^2 g_\rho^2 \left( 1+  \int \frac{y}{(\rho-y)^2} d \Pi (y)\right) -\frac{1}{g_\rho^2} \int \frac{1}{(\rho-y)^2} d \Pi (y) <0$$
and 
\begin{eqnarray*}v_\rho&=&  -\frac{3}{5}  \left( \theta^2 g_\rho +2 \right)^2 \\&+&
\theta^4g_\rho^4  \int\frac{y^2}{(\rho-y)^2 }d\mu_{\Pi}(y) + 2\theta^2 ( 2 +\theta^2 g_\rho^2 +g_\rho)  \int\frac{y}{(\rho-y)^2 }d\mu_{\Pi}(y)\\&&+ 
(\frac{1}{g_\rho}+\theta^2)^2  \int\frac{1}{(\rho-y)^2 }d\mu_{\Pi}(y) + 2\theta^4 g_\rho^4  \int\frac{y}{(\rho-y) }d\mu_{\Pi}(y)\\&&+ \theta^2 g_\rho^2( \theta^2 g_\rho^2 + 2\theta^2 g_\rho +2),
\end{eqnarray*}
with $$ \int\frac{y}{(\rho-y) }d\mu_{\Pi}(y)=-1+\rho g_\rho,$$
$$ \int\frac{y}{(\rho-y)^2 }d\mu_{\Pi}(y) =-g_\rho -\rho g^{'}_\rho,$$
$$ \int\frac{y^2}{(\rho-y)^2 }d\mu_{\Pi}(y)=1-2\rho g_\rho -\rho^2 g^{'}_\rho,$$
and $g^{'}_\rho =G_\Pi^{'}(\rho)=\frac{g_\rho(1-g_\rho)}{\rho(2g_\rho-1)}$ (after differentiating the equation $t G_\Pi(t)^2 -t G_\pi(t) +1=0$).
Thus, $$C_\rho^{(1)}= -\theta^4 g_\rho^4  + \frac{g_\rho^{'}}{g_\rho^2}(g_\rho +1), \;  C_\rho^{(2)}=-2\theta,$$  $$ v_\rho=  -\frac{3}{5} \left( \theta^2 g_\rho +2 \right)^2  -\frac{g_\rho^{'}}{g_\rho^2} \left(1+\frac{7}{g_\rho} +\theta^2 \right) -4 \theta^2 g_\rho.$$
Now, set $$C=\left|\frac{C_\rho^{(2)}}{C_\rho^{(1)}}\right|, \; \sigma^2= \frac{v_\rho}{(C_\rho^{(1)})^2}.$$
According to Theorem \ref{principal}, $\sqrt{N} ( \lambda(N,\rho)- \rho) $  converges in distribution to the probability measure with density function
$$f(x)= \frac{1}{2\sqrt{6\pi} C  \sigma} \int_{-\sqrt{3}  C }^{\sqrt{3}  C} \exp \left(-(x-t)^2/ 2\sigma^2\right) dt.$$
\section{Preliminary results}\label{prelim}
\subsection{\emph{Basic bounds and convergences}}
We start with  straightforward bounds and convergences involving resolvents and that will be of basic use for the proof of Theorem \ref{principal}.
For any $w \in M_m(\mathbb{C})$ such that $w \otimes I_{N-1} -\alpha \otimes \frac{W_{N-1}}{\sqrt{N}} -\beta \otimes A_{N-1}$, resp. $
w\otimes 1_{\cal A} -\alpha \otimes x -\beta\otimes a$, is invertible, define 
$$ R_{N-1}(w)=\left(w \otimes I_{N-1} -\alpha \otimes \frac{W_{N-1}}{\sqrt{N}} -\beta \otimes A_{N-1}\right)^{-1}, $$ resp. $$R_\infty(w) = (w\otimes 1_{\cal A} -\alpha \otimes x -\beta\otimes a)^{-1}.$$
Note that we have the following resolvent identities for any  $w_1$ and $w_2$ in $M_m(\mathbb{C})$ such that the resolvents are defined:
\begin{equation}\label{residN} R_{N-1}(w_1 )- R_{N-1}(w_2) = R_{N-1}(w_1) \left[(w_2-w_1) \otimes I_{N-1} \right]R_{N-1}(w_2),\end{equation}
\begin{equation}\label{residinfini} R_{\infty}(w_1 )- R_{\infty}(w_2) = R_{\infty}(w_1)\left[ (w_2-w_1) \otimes 1_{\cal A}\right] R_{\infty}(w_2).\end{equation}
\begin{lemma}\label{resolvante}
\begin{itemize}
\item For any $w \in M_m(\mathbb{C})$ such that $\Im w >0$, 
$w \otimes I_{N-1} -\alpha \otimes \frac{W_{N-1}}{\sqrt{N}} -\beta \otimes A_{N-1}$,   $
w\otimes 1_{\cal A} -\alpha \otimes x -\beta\otimes a$ and $w\otimes 1_{\cal A} -\alpha \otimes x -\beta\otimes a_{N-1}$ are  invertible and
\begin{equation}\label{HTN}\Vert R_{N-1}(w)\Vert \leq \Vert (\Im w)^{-1} \Vert,\end{equation}
\begin{equation}\label{HTinfini}\Vert R_{\infty}(w)\Vert \leq \Vert (\Im w)^{-1} \Vert,\end{equation}
\begin{equation}\label{HTinfiniN} \left\|\left( w\otimes 1_{\cal A} -\alpha \otimes x -\beta\otimes a_{N-1}\right)^{-1} \right\|\leq \Vert (\Im w)^{-1} \Vert.\end{equation}
\item 
Let $\tilde \Omega_{N-1}$ be as defined by \eqref{defOmegaN} and $C_\epsilon$ be as in \eqref{distanceinfini} and \eqref{distance}.
Let $z$ be in $\mathbb{R}$ such that $\vert z-\rho\vert< C_\epsilon/4$ and $z_0$  be in $\mathbb{C}$ such that $\vert z_0\vert< C_\epsilon/4$.\\
Then, 
  $
(ze_{11} + z_0 I_m-\gamma)\otimes 1_{\cal A} -\alpha \otimes x -\beta\otimes a$, $(ze_{11} + z_0 I_m-\gamma)\otimes 1_{\cal A} -\alpha \otimes x -\beta\otimes a_{N-1}$ and $(ze_{11} + z_0 I_m-\gamma)\otimes I_{N-1} -\alpha \otimes \frac{W_{N-1}}{\sqrt{N}} -\beta \otimes A_{N-1}$ on  $\tilde \Omega_{N-1}$,  are  invertible and 
we have
\begin{equation}\label{borneRN}\left\| R_{N-1}(ze_{11} + z_0 I_m-\gamma)\1_{\tilde \Omega_{N-1}}\right\| \leq 2/C_\epsilon,\end{equation} 
\begin{equation}\label{borneRinfini}\left\| R_{\infty}(ze_{11}+z_0 I_m-\gamma)\right\| \leq 2/C_\epsilon,\end{equation}
\begin{equation}\label{HTinfiniNpas} \left\|\left( (ze_{11}+z_0 I_m-\gamma)\otimes 1_{\cal A} -\alpha \otimes x -\beta\otimes a_{N-1}\right)^{-1} \right\|\leq 4/C_\epsilon.\end{equation}
Moreover,  for any $t$ in the spectrum of $a$, $ \omega_m(ze_{11}+z_0 I_m-\gamma) -t\beta$ is invertible and 
 {\begin{equation}\label{bornetilde} \left\| \left( \omega_m(ze_{11}+z_0 I_m-\gamma) -t\beta \right)^{-1} \right\| \leq  2/C_\epsilon, \end{equation}} and, for any $t$ in the spectrum of $A_{N-1}$, $ \omega_m^{(N)}(ze_{11}+z_0 I_m-\gamma) -t\beta$ is invertible and
  {\begin{equation}\label{bornetildeN}\left\| \left( \omega_m^{(N)}(ze_{11}+z_0 I_m-\gamma) -t\beta \right)^{-1} \right\| \leq  4/C_\epsilon. \end{equation}}
\end{itemize}
\end{lemma}
\begin{proof}
\eqref{HTN}, \eqref{HTinfini} and \eqref{HTinfiniN} come from Lemma 3.1 (i) of \cite{HT05}. Now, 
 according to \cite{M}, since  $A_{N-1}$ (and thus $a_{N-1}$) converges strongly to $a$, we have, for all large $N$,\\

\noindent 
$
\text{spect}((\rho e_{11}-\gamma) \otimes 1_{\cal A} - \alpha \otimes x -  \beta \otimes a_{N-1})$ \begin{equation} \label{spectre2} \subset
\text{spect}((\rho e_{11}-\gamma) \otimes 1_{\cal A} -  \alpha\otimes x - \beta \otimes a) + ]-C_{\epsilon}/4, C_\epsilon/4[.
\end{equation}
\eqref{borneRN}, \eqref{borneRinfini} and \eqref{HTinfiniNpas} easily follow from \eqref{distanceinfini}, \eqref{distance}, \eqref{spectre2} and the following facts: if $y$ is a
 self-adjoint element in a ${\cal C}^*$-algebra and $\lambda_0\in \mathbb{C}\setminus \text{spect}(y) $,  
then $d(\lambda_0, \text{spect}(y))= 1/ \Vert (\lambda_0 -y)^{-1} \Vert$
 and for any other  element  $\tilde y$, the distance between  any element in the spectrum of 
$\tilde y$ and the spectrum of 
$ y$  is smaller than $ \Vert y-\tilde y\Vert$.\\
Using the analyticity on the set $\left\{w \in M_m(\C), w\otimes 1_{\cal A} -\alpha\otimes x -\beta \otimes a \mbox{~invertible}\right\}$ of the functions  $E_{\cal N} \left[R_\infty(\cdot)\right]$ and $\omega_m$, 
one can easily deduce from \eqref{inversi} that,  \\

$I_m\otimes 1_{\cal A} $
\begin{eqnarray} &=& 
\left( \omega_m(ze_{11}+z_0 I_m-\gamma)\otimes 1_{\cal A} -\beta \otimes a \right)
E_{\cal N} \left[R_\infty\left(ze_{11}+z_0 I_m-\gamma \right)\right] \nonumber \\& =&  
E_{\cal N} \left[R_\infty\left(ze_{11}+z_0 I_m-\gamma \right)\right] \left( \omega_m(ze_{11}+z_0 I_m-\gamma)\otimes 1_{\cal A} -\beta \otimes a \right).\label{Efoisomega}\end{eqnarray}
Let $t$ be in the spectrum of $a$. Choose a character $\chi$  of
the commutative ${\cal C}^*$-algebra $\C<a>$ such that $\chi(a)=t$ and denote by $\chi_m : M_m(\C<a>) \rightarrow M_m(\C)$
 the
algebra homomorphism obtained by applying $\chi$ to each entry. Applying $\chi_m$ to \eqref{Efoisomega}, we deduce that 
$$\left( \omega_m(ze_{11}+z_0 I_m-\gamma) -t\beta \right)^{-1}=\chi_m(E_{\cal N} \left[R_\infty\left(ze_{11}+z_0 I_m-\gamma \right)\right])$$
 so that 
\eqref{bornetilde} readily follows from \eqref{borneRinfini}. \eqref{bornetildeN} can be proven similarly.
\end{proof}

 The following convergence results are quite straightforward consequences of asymptotic freeness of $W_{N-1}/\sqrt{N}$ and $A_{N-1}$.

\begin{lemma} \label{lemR}
For any  $\Sigma$, $\Sigma_1$, $\Sigma_2$  in $M_m(\mathbb C)$ such that $\exists C>0$, $\Vert \Sigma \Vert \leq C,\; \Vert \Sigma_1 \Vert \leq C, \;
\Vert \Sigma_2 \Vert \leq C$, almost surely,
 \begin{enumerate}
 \item[1)] $$ \left( {\rm id}_m  \otimes \tr_{N-1}\right)(R_{N-1}(\rho_N e_{11} - \gamma) \left(\Sigma\otimes I_{N-1} \right) R_{N-1}(\rho_N e_{11} - \gamma))\1_{\tilde \Omega_{N-1}} $$ \begin{equation} \label{CV1rhoN} \longrightarrow_{N\rightarrow +\infty}
\left( {\rm id}_m  \otimes \phi\right) \left[R_\infty(\rho e_{11} - \gamma)  \left(\Sigma \otimes 1_{\mathcal A}\right) R_\infty(\rho e_{11} - \gamma)\right] 
 \end{equation} 
\item[2)]
~~

$ \1_{\tilde \Omega_{N-1}}\times
\tr_{N-1}\left\{\left(\Tr_m\otimes {\rm id}_{N-1}\right)\left[R_{N-1}(\rho_N e_{11} - \gamma) \left(\Sigma_1\otimes I_{N-1}\right)\right]\right.$
$$~~~~~~~~~~~~~~~~~~~~~~~\hspace*{0.3cm}\times \left.\left(\Tr_m\otimes {\rm id}_{N-1}\right)\left[R_{N-1}(\rho_N e_{11} - \gamma) \left(\Sigma_2\otimes I_{N-1}\right)\right]\right\}$$
 $$ \longrightarrow_{N\rightarrow +\infty}  \phi\left\{ \left(\Tr_m\otimes {\rm id}_{\cal A}\right)\left[R_\infty(\rho e_{11} - \gamma) \left(\Sigma_1\otimes 1_{\mathcal A}\right)\right] \right.$$ \begin{equation}\label{CV2wrhoN}~\hspace*{1.3cm}~~~~~~~~~~~~~~~~~~~~~~~~~~~~~~\times\left.\left(\Tr_m\otimes {\rm id}_{\cal A}\right)\left[R_\infty(\rho e_{11} - \gamma)\left( \Sigma_2\otimes 1_{\mathcal A}\right)\right]\right\},
 \end{equation}
\item[3)]  \begin{equation}\label{convomegaplus} \forall w \in M_m(\C), \Im w >0, \; \omega_m^{(N)} (w) \longrightarrow_{N\rightarrow +\infty} \omega_m(w).\end{equation}
\begin{equation}\label{convomegaz} \forall z \in \R, \vert z-\rho\vert < \tau, \; \omega_m^{(N)} (ze_{11} -\gamma) \longrightarrow_{N\rightarrow +\infty} \omega_m(ze_{11} -\gamma).\end{equation}
\begin{equation}\label{convomega}\omega_m^{(N)} (\rho_N e_{11} -\gamma) \longrightarrow_{N\rightarrow +\infty} \omega_m(\rho e_{11} -\gamma).\end{equation}
 \end{enumerate}
 \end{lemma}
  \begin{proof}  
We have for any self-adjoint operators $u$ and $v$, for any $w\in M_m(\C)$ such that $\Im w >0$, for any non null integer $p$, \\

$   (w \otimes 1 -\alpha  \otimes  u -\beta  \otimes v )^{-1} $
\begin{eqnarray}&=& \sum_{k=0}^{p-1} w^{-1}\otimes1 (\alpha w^{-1} \otimes u +\beta w^{-1} \otimes v )^k \nonumber \\&&+    \left(w \otimes 1 -\alpha  \otimes u -\beta  \otimes v \right)^{-1} (\alpha w^{-1} \otimes u +\beta w^{-1}\otimes v )^p.\label{appol}\end{eqnarray}
 For any $K>0$, define $$\mathcal{O}_K = \{ w \in M_m(\mathbb C), \Im(w) > K I_m\}.$$
According to Lemma 3.1 (i) of \cite{HT05}, for any $w \in \mathcal{O}_K$, we have $\Vert w^{-1} \Vert \leq 1/K$.
Let $0<C<1$. For any $\kappa >0$,   there exists  $K=K(\kappa, \alpha, \beta)>0$  such that if $w \in {\cal O}_K$, for any $u$ and $v$ such that $\Vert u \Vert \leq \kappa$ and $\Vert v \Vert \leq \kappa $ then  \begin{equation}\label{defK} \Vert  (\alpha w^{-1} \otimes u +\beta w^{-1}\otimes v )\Vert \leq C,  \end{equation} 
 so that (using once more Lemma 3.1 (i) of \cite{HT05})  
$$\sup_{w \in {\cal O}_K} \left\|   \left(w\otimes 1 -\alpha  \otimes u -\beta  \otimes v \right)^{-1} (\alpha w^{-1} \otimes u +\beta w^{-1}\otimes v )^p  \right\| \leq \frac{C^p}{K} \rightarrow_{p \rightarrow +\infty} 0.$$
\noindent Fix $K>0$ such that \eqref{defK} holds for $(u,v)= (x,a)$ and $(u,v)=(\frac{W_{N-1}}{\sqrt{N}},A_{N-1})$ on $\tilde  \Omega_{N-1}.$ Therefore, for any $\delta>0$,  we can find a polynomial $Q_w$ with coefficients in $M_m(\mathbb{C})$ depending on $w, \alpha$ and $\beta$ such that:
\begin{equation} \label{approx} \sup_{w\in {\cal O}_K}|| R_{\infty}(w) - Q_w(x, a) || \leq \delta, \end{equation} \begin{equation} \label{approx2} \sup_{w\in {\cal O}_K}|| R_{N-1}(w) - Q_w\left(\frac{W_{N-1}}{\sqrt{N}}, A_{N-1}\right) || \1_{\tilde \Omega_{N-1}} \leq \delta.\end{equation} 
 Now, by the asymptotic freeness of $\frac{W_{N-1}}{\sqrt{N}}$ and $ A_{N-1}$ (see \cite[Theorem 5.4.5]{AGZ}), we have that almost surely
  $$
   \left({\rm id}_m  \otimes \tr_{N-1}\right)\left\{Q_w\left(\frac{W_{N-1}}{\sqrt{N}}, A_{N-1}\right) (\Sigma \otimes I_N )Q_w\left(\frac{W_{N-1}}{\sqrt{N}}, A_{N-1}\right)\right\}$$ \begin{equation}\label{Cvpol} \longrightarrow_{N \rightarrow \infty} 
\left( {\rm id}_m  \otimes \phi\right)\left\{Q_w(x, a) (\Sigma \otimes 1_{\cal A} )Q_w(x, a)\right\} .
 \end{equation}
 Using \eqref{Cvpol}, \eqref{approx}, \eqref{approx2}, \eqref{HTN}, \eqref{HTinfini} and 
\begin{equation} \label{smiley}\lim_{N\rightarrow +\infty} \1_{\tilde \Omega_{N-1}}=1 \; \text{a.s.},\end{equation} we can deduce 
that for $w\in  {\cal O}_K$, $\left({\rm id}_m  \otimes \tr_{N-1}\right)(R_{N-1}(w) (\Sigma\otimes I_{N-1}) R_{N-1}(w))$
  converges almost surely  towards $\left({\rm id}_m  \otimes \phi\right)(R_\infty(w) (\Sigma \otimes 1_{\mathcal A}) R_\infty(w))$, when $N$ goes to infinity. \\
Let $\mathcal{O} = \{ w \in M_m(\mathbb C), \Im(w) >0\}$.
 The two functions $$\Phi_N(w)  = \left({\rm id}_m  \otimes \tr_{N-1}\right)\left[R_{N-1}(w) \left( \Sigma\otimes I_{N-1} \right)R_{N-1}(w)\right] $$ and 
$$\Phi_\infty (w) = \left({\rm id}_m  \otimes \phi\right) \left[R_\infty(w) \left( \Sigma \otimes 1_{\mathcal A}\right) R_\infty(w)\right]$$ are holomorphic on $\mathcal{O}$. Moreover, using \eqref{HTN}, we have 
$$ || \Phi_N(w)|| \leq || (\Im w)^{-1}||^2  ||\Sigma||\leq C|| (\Im w)^{-1}||^2.$$
It readily follows that $\Phi_N$ is a bounded sequence in the set of analytic functions on ${\cal O}$ endowed with the uniform convergence on compact subsets. Since moreover, almost surely, for any $t >K$, $t\in \mathbb{Q}$, $\Phi_N(it I_m)$
converges towards $\Phi(it I_m)$,
we can apply Vitali's theorem to conclude that almost surely the convergence of $\Phi_N$ towards $\Phi_\infty$ holds on $ \mathcal{O}$. Of course, this convergence still holds on $-{\cal O}$. \\
Let $z\in \mathbb{R}$ be such that $\vert z-\rho\vert \leq C_\epsilon/4$.
For any $q >0$, such that $\frac{1}{ q} \leq C_\epsilon/4 $, define $w(q) = ze_{11} -\gamma +i \frac{1}{q} I_m$. 
  Almost surely, for any such $q$,   $\Phi_N(w(q))$ converges  towards $\Phi_\infty(w(q))$.
Using  \eqref{smiley}, the resolvent identities \eqref{residinfini} and \eqref{residN}
 on $\tilde \Omega_{N-1}$,
and the bounds \eqref{borneRN} and \eqref{borneRinfini},
we easily deduce by letting q goes to infinity that a.s. $$\left({\rm id}_m  \otimes \tr_{N-1}\right)\left[R_{N-1}(ze_{11}-\gamma) \left( \Sigma \otimes I_{N-1}\right)\ R_{N-1}(ze_{11}-\gamma)\right]\1_{\tilde  \Omega_{N-1}}$$ \begin{equation} \label{CV1z}\longrightarrow_{N\rightarrow +\infty} 
\left( {\rm id}_m  \otimes \phi\right) \left[R_\infty(ze_{11}-\gamma) \left( \Sigma \otimes 1_{\mathcal A}\right) R_\infty(ze_{11}-\gamma)\right].
 \end{equation}
 Note that  using \eqref{convrho}, the bound \eqref{borneRN},  
and the resolvent identity 
\eqref{residN} on $\tilde \Omega_{N-1}$, \eqref{CV1rhoN} readily follows from \eqref{CV1z} applied to $z=\rho$. \\
\eqref{CV2wrhoN}, \eqref{convomegaplus}, \eqref{convomegaz} and \eqref{convomega} can be proven using similar  ideas.
\end{proof}
The proof of Theorem \ref{principal}, that will be presented in Section \ref{Preuve}, is based on the writing of the outlier in terms of a quadratic form involving the resolvent $R_{N-1}$. 
Section \ref{TCLFQ} presents the central limit theorem for random  quadratic forms involved in the proof whereas Section \ref{neg} gather results 
 that will be used to prove that some terms  are negligible. 
\subsection{\emph{Central limit theorem for random quadratic forms}} \label{TCLFQ}
\begin{proposition}\label{Delta1} 
 For any Hermitian $m \times m$  matrix $H$,
$$\sqrt{N}\left( \frac{1}{N} \Tr_m\left\{H  \left(\alpha\otimes Y^*\right) R_{N-1}(\rho_N e_{11}-\gamma)\1_{\tilde \Omega_{N-1}} \left(\alpha \otimes Y\right)\right\}\right.$$ $$\left.~~~~ ~~-\frac{1}{N} \Tr_m\left\{\alpha H \alpha \left[\left({\rm id}_m \otimes \Tr_{N-1} \right)  \left(  R_{N-1}(\rho_N e_{11}-\gamma)\1_{\tilde \Omega_{N-1}}\right)\right]  \right\}\right)$$  converges in distribution to a Gaussian variable with mean 0 and variance 
  $$\left(\mathbb{E}\left(\vert W_{12}\vert^4\right)-2\right) \int \left[
  \Tr_m \left(\alpha H \alpha\left( \omega_m(\rho e_{11} -\gamma) -t\beta\right)^{-1}\right) \right]^2d\mu_a(t)
  $$
	$$ +\phi \left(\left[  \left(\Tr_m\otimes {\rm id}_{\cal A} \right) \left\{(\rho e_{11} -\gamma)\otimes 1_{\cal A}-\alpha\otimes x - \beta \otimes a)^{-1} \left((\alpha H \alpha)\otimes 1_{\cal A}\right) \right\}\right]^2\right).$$
\end{proposition}
\begin{proof}
  We  apply the following Proposition \ref{BYvaleurpropre}  to  $B=R_{N-1}(\rho_N e_{11}-\gamma)\1_{\tilde \Omega_{N-1}}$ by using \eqref{convrho}, \eqref{borneRN} and  Proposition  \ref{conditionsBaiYao} below. \end{proof}
	
 \begin{proposition}\label{BYvaleurpropre}
   Let $m$ be a fixed integer number and $\alpha$ be  a Hermitian $m\times m$ deterministic matrix that does not depend on $N$.  Let  $B$  be  a random  Hermitian $mN\times mN$   matrix such that there exists $C>0$ such that $\Vert B \Vert \leq C$.  Let us write $B=\sum_{i,j=1}^N B_{ij} \otimes E_{ij}$ where $B_{ij}$ are $m\times m$ matrices.
 Assume  that, for any $p,q,p',q'$ in $\{1,\ldots,m\}^4$, \begin{itemize}
   \item $\frac{1}{N} \sum_{i=1}^N (B_{ii})_{pq} (B_{ii})_{p'q'} \rightarrow_{N\rightarrow +\infty} \omega_{(p,q),(p',q')} \text{~in probability},$
   \item
	$\frac{1}{N} \sum_{i,j=1}^N (B_{ij})_{pq} (B_{ji})_{p'q'}\rightarrow_{N\rightarrow +\infty} \theta_{(p,q),(p',q')} \text{~in probability}.$
   \end{itemize} 
	Let $^tX= (x_1, \ldots , x_N)$ be an independent 
vector of size N  which contains i.i.d. complex  standardized entries with
bounded fourth moment and such that $\mathbb{E}(x_1^2)=0$. Let $H$ be a $m\times m $ deterministic Hermitian matrix that does not depend on $N$. Then, when $N$ goes to infinity, $$ \frac{1}{\sqrt{N}}  \Tr_m \left\{ H \left[  \left(\alpha \otimes X^* \right) B \left( \alpha \otimes X\right) -\alpha \left({\rm id}_m \otimes \Tr_N\right) ( B) \alpha \right]\right\}$$ converges in distribution to a Gaussian variable with mean 0 and variance 
  $$\left(\mathbb{E}\left(\vert x_1\vert^4\right)-2\right) \sum_{p,q,p',q'=1}^m (\alpha H \alpha)_{qp} 
	(\alpha H \alpha)_{q'p'}\omega_{(p,q),(p',q')} $$ $$+\sum_{p,q,p',q'=1}^m (\alpha H \alpha)_{qp} (\alpha H \alpha)_{q'p'}\theta_{(p,q),(p',q')}.$$

  \end{proposition}
  \begin{proof}  Note that 
  $$\frac{1}{\sqrt{N}} \Tr_m \left\{H\left[ \left(\alpha \otimes X^*\right)  B \left(  \alpha \otimes X \right)  -\alpha \left( {\rm id}_m \otimes \Tr_N\right) (B) \alpha\right]\right\}=  \frac{1}{\sqrt{N}}\left\{ X^* {\cal B}X -\Tr_N {\cal B}\right\} $$ where ${\cal B}=\left({\cal B}_{ij}\right)_{1\leq i,j\leq N}$ and $ {\cal B}_{ij}= \Tr_m \alpha H \alpha B_{ij}$.\\
 Thus, the result follows from  \cite{BY} or Theorem 5.2 in \cite{CDF}.

 \end{proof}
\begin{proposition}\label{conditionsBaiYaoprelim} When it is defined, let us rewrite 
	$$R_{N-1}= \sum_{i,j=1}^{N-1} (R_{N-1})_{ij}\otimes E_{ij},$$
	where $(R_{N-1})_{ij} \in M_m(\mathbb{C})$. 
For any $w \in \mathbb{H}^+_m(\C)$, we have that, almost surely,
	\begin{equation}\label{defFN} F_N(w)=\frac{1}{N-1} \sum_{i=1}^{N-1} [(R_{N-1}(w-\gamma))_{ii}]_{pq} [(R_{N-1}(w-\gamma)_{ii})]_{p'q'} \end{equation}
	\begin{equation}\label{cvFN} \rightarrow_{N\rightarrow +\infty} \int [(\omega_m( w-\gamma)-t \beta )^{-1}]_{pq} [(\omega_m( w-\gamma)-t \beta )^{-1}]_{p'q'}d\mu_a(t) \end{equation}

	\end{proposition}
 \begin{proof} 
First we are going to prove that almost surely, \\

$\frac{1}{N-1}\displaystyle{\sum_{i=1}^{N-1} [(R_{N-1}(w-\gamma))_{ii}]_{pq}[(R_{N-1}(w-\gamma))_{ii}]_{p'q'}}$ \begin{equation}\label{concentrationFN}-
\frac{1}{N-1}\sum_{i=1}^{N-1}[\mathbb{E}(R_{N-1}(w-\gamma))_{ii}]_{pq}[\mathbb{E}(R_{N-1}(w-\gamma))_{ii}]_{p'q'} \longrightarrow_{N \rightarrow+\infty} 0.\end{equation}
Set $a_i=[(R_{N-1}(w-\gamma))_{ii}]_{pq}$ and $b_i=[(R_{N-1}(w-\gamma))_{ii}]_{p'q'}$.
We have\\ 

$\displaystyle{\frac{1}{N-1}\sum_{i=1}^{N-1} a_i b_i - \frac{1}{N-1}\sum_{i=1}^{N-1} \mathbb{E}(a_i)\mathbb{E}(b_i)}$
\begin{eqnarray*}&=& \frac{1}{N-1}\sum_{i=1}^{N-1} a_i b_i - \frac{1}{N-1}\sum_{i=1}^{N-1} \mathbb{E}(a_ib_i)\\
&&+\frac{1}{N-1}\sum_{i=1}^{N-1} \mathbb{E}\left\{ (a_i- \mathbb{E}(a_i))( b_i -\mathbb{E}(b_i))\right\}.
\end{eqnarray*} 
Consider the linear isomorphism $\Psi$ between $M_{N-1}^{sa}(\mathbb{C})$ and $\mathbb{R}^{{(N-1)}^2}$ given by 
 \begin{equation}\label{defiso}\Psi ((a_{kl})_{1\leq k,l\leq {N-1}}) =\left( (a_{kk})_{1\leq k \leq {N-1}}, (\sqrt{2}\Re a_{kl})_{1\leq k<l\leq {N-1}}, (\sqrt{2}\Im a_{kl})_{1\leq k<l\leq {N-1}}\right)\end{equation}
 for $(a_{kl})_{1\leq k,l\leq {N-1}}$ in $M^{sa}_{N-1}(\mathbb{C})$.
 $M_{N-1}(\mathbb{C})^{sa}$  is an Euclidean space with inner product given by $\langle A,B \rangle=\Tr_{N-1}(AB)$ and with norm 
 $$\Vert A\Vert_e=\left( \Tr_{N-1} A^2\right)^{1/2}.$$ We shall identify $M_{N-1}^{sa}(\mathbb{C})$ with $R^{({N-1})^2}$ via the isomorphism $\Psi$. Note that under this identification the norm $\Vert \cdot \Vert_e$ on $M_{N-1}^{sa}(\mathbb{C})$ corresponds to the usual Euclidean norm on $R^{({N-1})^2}$.\\

\noindent Define $f_N: M_{N-1}^{sa}(\mathbb{C}) \mapsto \C$ by 
 $$f_{N}(W) = \frac{1}{N-1}\sum_{i=1}^{N-1}\Tr_{m(N-1)} \left[ \left((w-\gamma) \otimes I_{N-1}   -\alpha  \otimes {W} -\beta \otimes A_{N-1} \right)^{-1}\left( e_{qp} \otimes E_{ii}\right)\right]$$$$\times  \Tr_{m(N-1)} \left[ \left((w-\gamma)\otimes I_{N-1}   -\alpha  \otimes {W} -\beta \otimes A_{N-1} \right)^{-1} \left(e_{q'p'} \otimes E_{ii}\right)\right].$$
Using the resolvent identity, for  $H_1, H_2 \in M_{m(N-1)}^{sa}(\mathbb{C})$,\\

$ \left(w \otimes I_{N-1}   -H_1\right)^{-1}-\left(w \otimes I_{N-1}   -H_2\right)^{-1}$ \begin{equation}\label{pourlip}=\left(w \otimes I_{N-1}   -H_1\right)^{-1} \left(H_1-H_2\right)
\left(w \otimes I_{N-1}   -H_2\right)^{-1}, \end{equation} and \cite[Lemma 3.1 (i)]{HT05}, one can easily prove that 
$f_N \circ \Psi^{-1}$ is Lipschitz with constant $\Vert(\Im w)^{-1} \Vert^3 $. Therefore, according to Lemma \ref{Herbst},  
 $$\mathbb{P}\left(\left| \frac{1}{N-1}\sum_{i=1}^{N-1} \left(a_ib_i -\mathbb{E}(a_ib_i)\right)\right|>\varepsilon\right)\leq K_1\exp\left(-K_2 N^{1/2}\Vert(\Im w)^{-1} \Vert^{-3}\varepsilon\right).$$ By Borell-Cantelli lemma, we deduce that, almost surely, when $N$ goes to infinity, $ \frac{1}{N-1}\sum_{i=1}^{N-1} a_i b_i - \frac{1}{N-1}\sum_{i=1}^{N-1} \mathbb{E}(a_ib_i)$ goes to zero.\\
Now define $g_N: M_{N-1}^{sa}(\mathbb{C}) \mapsto \C$ by 
 $$g_{N}(W) = \Tr_{m(N-1)} \left[ \left((w-\gamma) \otimes I_{N-1}   -\alpha  \otimes {W} -\beta \otimes A_{N-1} \right)^{-1} \left(e_{qp} \otimes E_{ii}\right)\right].$$
Define also $\tilde g_N: \mathbb{R}^{{(N-1)}^2} \rightarrow \mathbb{C}$ by $\tilde g_N=g_N\circ \Psi^{-1}$, where $\Psi$ is defined in \eqref{defiso}. Note that 
 $$\left\| \nabla \tilde g_N(\Psi(W)) \right\|= \left\| \text{grad} g_N(W)\right\|_e$$
 and 
 $$ \left\| \text{grad} g_N(W)\right\|_e^2 =\sup_{w\in S_1(M_{N-1}^{sa}(\mathbb{C}))} \left| \frac{d}{dt} g_N(W+tw)_{\vert_{ t=0}}\right|^2,$$
 where $S_1(M_{N-1}^{sa}(\mathbb{C}))$ denotes the unit sphere of $M_{N-1}^{sa}(\mathbb{C})$ with respect to $\Vert \cdot \Vert_e$.
 Applying Poincar\'e inequality for $\tilde g_N$, we get that 
 $$\mathbb{E}\left( \left| g_N\left(\frac{W_{N-1}}{\sqrt{N}}\right) -\mathbb{E}\left\{g_N\left(\frac{W_{N-1}}{\sqrt{N}}\right)\right\}\right|^2 \right) \leq \frac{C}{N} \mathbb{E}\left( \left\| \text{grad} g_N\left(\frac{W_{N-1}}{\sqrt{N}}\right)\right\|_e^2 \right).$$
Using \eqref{pourlip} and \eqref{HTN}, it readily follows that, there exists $C>0$, such that for any $i=1,\ldots, N-1$, $$\mathbb{E}\left| a_i-  \mathbb{E}(a_i) \right|^2 \leq \frac{C \Vert (\Im w)^{-1}\Vert^4}{N}$$ and similarly  $$\mathbb{E}\left| b_i-  \mathbb{E}(b_i) \right|^2 \leq \frac{C \Vert (\Im w)^{-1}\Vert^4}{N}$$ so that $\frac{1}{N-1}\sum_{i=1}^{N-1} \mathbb{E}\left\{ (a_i- \mathbb{E}(a_i))( b_i -\mathbb{E}(b_i))\right\}$ goes to zero as $N$ goes to infinity. Thus, the proof of \eqref{concentrationFN} is complete. \\

\begin{lemma}\label{estimR} For any $w \in \mathbb{H}^+_m(\C)$, for any $j\in \{1,\ldots, N-1\}$,
$$\mathbb{E}\left\{(R_{N-1}(w-\gamma))_{jj}\right\}= (\omega_m^{(N)}(w-\gamma) -d_j \beta)^{-1} +O_{j}^{(u)}(1/\sqrt{N}).$$
\end{lemma}
\begin{proof} First set 
$$\hat R_{N-1} (w)=\left( w\otimes I_{N-1}  -\alpha \otimes\frac{ W_{N-1}}{\sqrt{N-1}}- \beta \otimes A_{N-1} \right)^{-1}.$$
Using Lemma 3.1 (i) of \cite{HT05}, we have \begin{equation}\label{chapeau}
\Vert \hat R_{N-1}(w) \Vert \leq \Vert (\Im w)^{-1} \Vert.\end{equation}
Note that, 
\begin{eqnarray*}R_{N-1}(w)&=&\hat R_{N-1}(w)\\ &&+\frac{1}{\sqrt{N-1}(\sqrt{N} +\sqrt{N-1})}\\&& ~~~~~~~\times\left(I_m\otimes I_{N-1} - R_{N-1}(w)  \left( w\otimes I_{N-1} - \beta \otimes A_{N-1} \right)\right)
\hat R_{N-1}(w).\end{eqnarray*}
Thus, using \eqref{chapeau}, \eqref{HTN} and \eqref{sup}, it readily follows that for any $w \in \mathbb{H}^+_m(\C)$, for any $j\in \{1,\ldots, N-1\}$,
\begin{equation}\label{correction} 
\mathbb{E}\left\{(R_{N-1}(w-\gamma))_{jj}\right\}=\mathbb{E}\left\{(\hat R_{N-1}(w-\gamma))_{jj}\right\} +O_{j}^{(u)}(1/{N}).\end{equation}
Therefore, in the following, we will prove that $$\mathbb{E}\left\{(\hat R_{N-1}(w-\gamma))_{jj}\right\}= (\omega_m^{(N)}(w-\gamma) -d_j \beta)^{-1} +O_{j}^{(u)}(1/\sqrt{N}).$$
Denote by $\kappa_3$  the classical third cumulant
of
$\mu$.
According to Corollary 5.5 in \cite{BC}, for any $j\in\{1,\ldots, N-1\}$, \\

 $\mathbb{E}\left\{(\hat R_{N-1}(w-\gamma))_{jj}\right\}$
\begin{eqnarray*}&= &(Y_{N-1}(w))_{jj}\\&&+ \sum_{i,l=1}^{N-1}  \frac{\kappa_3 (1-\sqrt{-1})}{2\sqrt{2} (N-1) \sqrt{N-1}} (Y_{N-1}(w))_{jl} \alpha (Y_{N-1}(w))_{ii}  \alpha (Y_{N-1}(w))_{ll} \alpha \\&&~~~~~~\times \mathbb{E}\left\{(\hat R_{N-1}(w-\gamma))_{ij}\right\} +O_{j}^{(u)}(1/N),\end{eqnarray*}
where  \begin{equation}\label{premieroubli} Y_{N-1}(w)=\left((w -\gamma -\alpha G_{N-1}(w)\alpha)\otimes I_{N-1} -  \beta \otimes A_{N-1} \right)^{-1}\end{equation}
with $$G_{N-1}(w)= \left(id_m \otimes \tr_{N-1}\right) \left( \hat R_{N-1}(w-\gamma)\right) .$$
Note that according to \cite[(5.7)]{BC}, $\Im \left[(w -\gamma -\alpha G_{N-1}(w)\alpha)\right] \geq \Im w$ so that, indeed, by Lemma 3.1  of \cite{HT05}, $
(w -\gamma -\alpha G_{N-1}(w)\alpha)\otimes I_{N-1} -  \beta \otimes A_{N-1}$ is  invertible and 
we have 
\begin{equation}\label{borneYN}\Vert Y_{N-1}(w) \Vert \leq \Vert (\Im w)^{-1} \Vert.\end{equation}
Now
 set
\begin{equation}\label{oubli}\tilde G_{N-1}(w)= id_m \otimes \phi \left( (w -\gamma )\otimes 1_{\cal A}  -\alpha \otimes x -\beta \otimes a_{N-1} \right)^{-1}.\end{equation}
Similarly, $(w -\gamma -\alpha \tilde G_{N-1}(w)\alpha)\otimes I_{N-1} -  \beta \otimes A_{N-1}$ is  invertible, we can define
    \begin{equation}\label{ytilde}\tilde Y_{N-1}(w)=\left((w -\gamma -\alpha \tilde G_{N-1}(w)\alpha)\otimes I_{N-1} -   \beta \otimes A_{N-1} \right)^{-1}\end{equation}
and 
we have 
\begin{equation}\label{bornetildeYN}\Vert \tilde  Y_{N-1}(w) \Vert \leq \Vert (\Im w)^{-1} \Vert.\end{equation}
Using the resolvent identity, \eqref{borneYN}, \eqref{bornetildeYN} and \cite[(5.48)]{BC}, one can easily deduce that 
there exists a polynomial $Q$ with nonnegative coefficients such that, for any $w \in M_m(\mathbb{C})$ such that $\Im w >0$, $$\left\|Y_{N-1}(w)- \tilde Y_{N-1}(w)\right\| \leq \frac{Q(\Vert (\Im w)^{-1} \Vert }{\sqrt{N}}.$$  Note that 
$\tilde Y_{N-1}(w) =\left( \omega_m^{(N)}(w-\gamma) \otimes  I_{N-1} -\beta \otimes A_{N-1}\right)^{-1}.$
Now \\

$\displaystyle{\left\| \sum_{i,l=1}^{N-1}  \frac{\kappa_3 (1-\sqrt{-1})}{2\sqrt{2} (N-1) \sqrt{N-1}} (Y_{N-1})_{jl}  \alpha (Y_{N-1})_{ii} \alpha (Y_{N-1})_{ll} \alpha \mathbb{E}\left\{(\hat R_{N-1}(w-\gamma))_{ij}\right\}\right\|}$ \begin{eqnarray*}& \leq &\frac{C \Vert (\Im w)^{-1} \Vert^2}{\sqrt{N}} \left(\sum_{l=1}^{N-1} \Vert (Y_{N-1})_{jl} \Vert^2\right)^{1/2} \left(\sum_{i=1}^{N-1} \left\| \mathbb{E}\left\{(\hat R_{N-1}(w-\gamma))_{ij}\right\}  \right\|^2\right)^{1/2}\\&\leq & \frac{C m\Vert (\Im w)^{-1} \Vert^2}{\sqrt{N}}  \Vert Y_{N-1} \Vert \left\| \mathbb{E}\left(\hat R_{N-1}(w-\gamma)\right) \right\|\\
&\leq & \frac{C m\Vert (\Im w)^{-1} \Vert^4}{\sqrt{N}}, \end{eqnarray*}
where we used \cite[Lemma 8.1]{BC}, \eqref{borneYN} and \eqref{HTN}.
It readily follows that, for any $j \in \{1,\ldots, N-1\}$, \begin{eqnarray*}\mathbb{E}\left\{(\hat R_{N-1}(w-\gamma))_{jj}\right\}&=& \left(\left( \omega_m^{(N)}(w-\gamma) \otimes  I_{N-1} -\beta \otimes A_{N-1}\right)^{-1}\right)_{jj}\\&& + O_{j}^{(u)}(1/\sqrt{N}).\end{eqnarray*}
Now, note that 
 there exist two permutation matrices $\Pi_1$ and $\Pi_2$ in $M_{(N-1)m}(\mathbb{C})$ such that, for any matrices $A\in M_m(\mathbb{C})$, 
  $B\in M_{N-1}(\mathbb{C})$, $A\otimes B =\Pi_1 (B\otimes A) \Pi_2$.
	Therefore 
  \\
	
	$\left[\left(\left(\omega_m^{(N)}(w-\gamma) \otimes  I_{N-1} -\beta \otimes A_{N-1}\right)^{-1}\right)_{jj}\right]_{pq}$
  \begin{eqnarray*}
  &=& \Tr_{m(N-1)} \left[\left(\omega_m^{(N)}(w-\gamma) \otimes  I_{N-1} -\beta \otimes A_{N-1}\right)^{-1} \left(e_{qp}\otimes {E}_{jj}\right)\right]\\&=&\Tr_{m(N-1)} \left[\Pi_2^{-1}\left(   I_{N-1} \otimes \omega_m^{(N)}(w-\gamma) - A_{N-1}  \otimes\beta\right)^{-1}\Pi_1^{-1} \Pi_1\left({E}_{jj}\otimes e_{qp} \right) \Pi_2 \right]\\
 &=&\Tr_{m(N-1)} \left[\left(   I_{N-1} \otimes \omega_m^{(N)}(w-\gamma) - A_{N-1}  \otimes\beta \right)^{-1}\left({E}_{jj}\otimes e_{qp} \right)\right]\\
&=&\left[\left( \omega_m^{(N)}(w-\gamma)  -d_j\beta \right)^{-1}\right]_{pq}.
\end{eqnarray*}
	Thus, $$\mathbb{E}\left\{ (\hat R_{N-1}(w-\gamma))_{jj}\right\}= \left( \omega_m^{(N)}(w-\gamma)  -d_j\beta \right)^{-1} + O_{j}^{(u)}(1/\sqrt{N}).$$
 Lemma \ref{estimR} follows.
	
\end{proof}
Note that, using \eqref{imw} and \cite[Lemma 3.1 (i)]{HT05}, we have that for any $w \in H^+_m(\mathbb{C})$
\begin{equation}\label{borneomegamN} \left\| (\omega_m^{(N)}(w) -d_i \beta)^{-1}\right\| \leq \Vert (\Im w)^{-1} \Vert,\end{equation}
\begin{equation}\label{borneomegam} \left\| (\omega_m(w) -d_i \beta)^{-1}\right\| \leq \Vert (\Im w)^{-1} \Vert,\end{equation}
and then \\

$  \left\| (\omega_m^{(N)}(w) -d_i \beta)^{-1}- (\omega_m(w) -d_i \beta)^{-1}\right\| $
\begin{eqnarray}&\leq&
\left\| (\omega_m^{(N)}(w) -d_i \beta)^{-1}\left[ \omega_m(w) -\omega_m^{(N)}(w) \right] (\omega_m(w) -d_i \beta)^{-1}\right\| \nonumber \\&\leq& \Vert (\Im w)^{-1} \Vert^2 
\left\| \omega_m^{(N)}(w) - \omega_m(w) \right\|.\label{difdesomega}\end{eqnarray}
\noindent Recall that $F_N$ was defined by \eqref{defFN}.
  Lemma \ref{estimR} and \eqref{concentrationFN} yield that   for any $w  \in H^+_m(\mathbb{C})$,  almost surely,
  \begin{eqnarray*}
F_N (w)&= & \frac{ 1}{N-1}\sum_{i=1}^{N-1} [(\omega_m^{(N)}(w-\gamma) -d_i \beta)^{-1}]_{pq} [(\omega_m^{(N)}(w-\gamma) -d_i \beta)^{-1}]_{p'q'} +o(1)\\ &=&\frac{ 1}{N-1}\sum_{i=1}^{N-1} [(\omega_m(w-\gamma) -d_i \beta)^{-1}]_{pq} [(\omega_m(w-\gamma) -d_i \beta)^{-1}]_{p'q'} +o(1), \end{eqnarray*} using 
 \eqref{borneomegamN} in the first line and \eqref{borneomegamN}, \eqref{borneomegam}, \eqref{difdesomega} and \eqref{convomegaplus} in the last line. Thus $$F_N (w)=\int [(\omega_m(w-\gamma) -t \beta)^{-1}]_{pq} [(\omega_m(w-\gamma) -t \beta)^{-1}]_{p'q'}d\mu_{A_{N-1}}(t) +o(1)$$ where $\mu_{A_{N-1}}=\frac{1}{N-1} \sum_{i=1}^{N-1} \delta_{\lambda_i(A_{N-1})}$ is the empirical spectral measure of $A_{N-1}$. Since $\mu_{A_{N-1}}$ weakly converges towards $\mu_a$, Proposition \ref{conditionsBaiYaoprelim} follows. 
\end{proof}
 \begin{proposition}\label{conditionsBaiYao} When it is defined, let us rewrite 
	$$R_{N-1}= \sum_{i,j=1}^{N-1} (R_{N-1})_{ij}\otimes E_{ij},$$
	where $(R_{N-1})_{ij} \in M_m(\mathbb{C})$. We have that, almost surely, 
$$\frac{1}{N-1} \sum_{i=1}^{N-1}[(R_{N-1}(\rho_N e_{11}-\gamma))_{ii}]_{pq} [(R_{N-1}(\rho_N e_{11}-\gamma)_{ii})]_{p'q'}\1_{\tilde \Omega_{N-1}}$$\begin{equation}\label{prop3CV1bis} \rightarrow_{N\rightarrow +\infty} \int [(\omega_m( \rho e_{11}-\gamma)-t \beta )^{-1}]_{pq} [(\omega_m( \rho e_{11}-\gamma)-t \beta )^{-1}]_{p'q'}d\mu_a(t) \end{equation} and \\

$\displaystyle{ \frac{1}{N-1}\sum_{i,j=1}^{N-1} [(R_{N-1}(\rho_N e_{11}-\gamma))_{ij}]_{pq}[(R_{N-1}(\rho_N e_{11}-\gamma))_{ji}]_{p'q'}\1_{\tilde \Omega_{N-1}} \rightarrow_{N\rightarrow +\infty}}$ \begin{equation}\label{prop3CV2bis} \phi
\left\{\left(\Tr_m\otimes {\rm id}_{\cal A}\right)\left[R_{\infty}(\rho e_{11}-\gamma)\left(e_{qp}\otimes 1_{\cal A}\right)\right] \left(\Tr_m\otimes {\rm id}_{\cal A}\right)\left[R_{\infty}(\rho e_{11}-\gamma)\left(e_{q'p'}\otimes 1_{\cal A}\right)\right] \right\}.\end{equation}
 \end{proposition} 
\begin{proof}
First, with $w=\rho_N e_{11}-\gamma$,  
let us rewrite\\

$\frac{1}{N-1}\sum_{i,j=1}^{N-1} \{[R_{N-1}(w)]_{ij}\}_{pq}\{[R_{N-1}(w)]_{ji}\}_{p'q'}=$
$$\tr_{N-1} \left\{\left(\Tr_m\otimes {\rm id}_{N-1}\right)\left[R_{N-1}(w)\left(e_{qp}\otimes I_{N-1}\right)\right] \left(\Tr_m\otimes {\rm id}_{N-1}\right)\left[R_{N-1}(w)\left(e_{q'p'}\otimes I_{N-1}\right)\right] \right\}.$$ 
Thus 
\eqref{prop3CV2bis} readily follows from Lemma \ref{lemR}.\\

Now, according to Lemma \ref{resolvante},
on $\tilde \Omega_{N-1}$, $F_N$ defined by \eqref{defFN} is well defined at the points $w=ze_{11}, ze_{11}+i\frac{1}{r}$, for any $r\in \mathbb{Q}\setminus \{0\}$,  $0<1/r< \tau$ and  any $z \in \R$ such that $\vert z-\rho\vert < \tau$.     Using the bounds 
\eqref{borneRN}, \eqref{borneRinfini}, \eqref{bornetilde}
 and the resolvent identities \eqref{residN}, \eqref{residinfini}, one can easily prove that   \\

\noindent $\left|  F_N (ze_{11})\1_{\tilde \Omega_{N-1}} -\int [(\omega_m( ze_{11}-\gamma)-t \beta )^{-1}]_{pq} [(\omega_m( ze_{11}-\gamma)-t \beta )^{-1}]_{p'q'}d\mu_a(t) \right|$
\begin{eqnarray*}
&\leq & \frac{1}{r}\frac{32}{C_\epsilon^3} \left\{1+ \frac{2}{C_\epsilon^2}\Vert \alpha\Vert^2 \right\}  + \frac{4}{C_\epsilon^2}\1_{^c \tilde \Omega_{N-1}}\\
&&+\left| F_N(ze_{11}+i\frac{1}{r} I_m) \right.\\&&~~~~- \int [(\omega_m( ze_{11}+i\frac{1}{r}I_m-\gamma)-t \beta )^{-1}]_{pq} \\&&~~~~~~~~~~~~~~~~~~~~~~~~~~~~~~\times \left.[(\omega_m(ze_{11}
+i\frac{1}{r}I_m-\gamma)-t \beta )^{-1}]_{p'q'}d\mu_a(t)\right|.
\end{eqnarray*}
We deduce by letting $N$ go to infinity, using Proposition \ref{conditionsBaiYaoprelim}, and then $r$ go to infinity that  for any $z \in \R$ such that for   $\vert z-\rho\vert < \tau$, almost surely, $ F_N (ze_{11})\1_{\tilde \Omega_{N-1}} $ converges to $ \int [(\omega_m( ze_{11}-\gamma)-t \beta )^{-1}]_{pq} [(\omega_m( ze_{11}-\gamma)-t \beta )^{-1}]_{p'q'}d\mu_a(t)$ when $N$ goes to infinity.\\
Note that  using \eqref{convrho}, the resolvent identity \eqref{residN} on $\tilde \Omega_{N-1}$,
and the bound \eqref{borneRN},   
\eqref{prop3CV1bis} follows from  the result for $\rho$ instead of $\rho_N$. The proof of Proposition \ref{conditionsBaiYao} is complete.
 \end{proof}
\subsection{\emph{Basic technical results of negligeability}}\label{neg}
\begin{lemma}\label{loiGN}
For any $N$, let
   $X_N=\begin{pmatrix} x_1\\ \vdots \\ x_N \end{pmatrix}$ be random in  $\mathbb{C}^N$ with iid  standardized entries (~$\mathbb{E}(x_i)=0$, $\mathbb{E}( \vert x_i \vert^2)=1$, $\mathbb{E}(  x_i ^2)=0$) and $\mathbb{E}( \vert x_i \vert^4)< \infty$.
 Let $m$ be a fixed integer number and $\alpha$ be  a Hermitian $m\times m$ deterministic matrix.  Let  $B$  be  a Hermitian $mN\times mN$  independent  matrix such that $\sup_{N} \Vert B\Vert \leq C.$ Then 
   $$\frac{1}{N}\left(I_m \otimes X_N^*\right) B \left(I_m \otimes  X_N\right) -\left({\rm id}_m \otimes \tr_N\right)B =o_{\mathbb{P}}(1).$$
\end{lemma}
\begin{proof}
 Let us write $B=\sum_{p,q=1}^m e_{pq} \otimes B^{(pq)}$ where $ B^{(pq)}$ are $N\times N$ matrices.
 Noting that \\ 
 
 $ \frac{1}{N}  \left(I_m \otimes X_N^*\right) B \left(I_m \otimes  X_N\right) -{\rm id}_m \otimes \tr_NB$ $$= \frac{1}{N} \sum_{p,q=1}^m e_{pq} \left\{X_N^*B^{(pq)}    X_N -\Tr_N(B^{(pq)})\right\}, $$
 the result readily follows from Lemma 2.7 in \cite{BS}.
\end{proof}
\begin{lemma}\label{BC}  
For any $w \in H^+_m(\mathbb{C})$, \\

\noindent $  \left({\rm id}_m\otimes tr_{N-1}\right)\mathbb{E}\left[ R_{N-1} (w-\gamma)\right]$ $$=
  \left({\rm id}_m\otimes \phi\right) \left(\left((w-\gamma)\otimes 1_{\cal A} -\alpha\otimes x-\beta \otimes a_{N-1} \right)^{-1}\right) + O(1/N).$$
 This result still holds for $w\in M_m(\C)$ such that  $\Im w <0$.
\end{lemma}
\begin{proof}  By \eqref{correction}, it is sufficient to prove that \\

 $  \left({\rm id}_m\otimes tr_{N-1}\right)\mathbb{E}\left[ \hat R_{N-1} (w-\gamma)\right]$ $$=
  \left({\rm id}_m\otimes \phi\right) \left(\left((w-\gamma)\otimes 1_{\cal A} -\alpha\otimes x-\beta \otimes a_{N-1} \right)^{-1}\right) + O(1/N).$$
 According to Theorem 5.7 in \cite{BC}, we have \\

\noindent $
\left({\rm id}_m\otimes \tr_{N-1}\right)\mathbb{E}\left[ \hat R_{N-1} (w-\gamma)\right]$\begin{equation} \label{prediff}
-\left({\rm id_m}\otimes \phi \right) \left(\left((w-\gamma)\otimes 1_{\cal A} -\alpha\otimes x-\beta \otimes a_{N-1}\right)^{-1}\right)+{E_{N-1}(w)}= O(\frac{1}{N\sqrt{N}}),
\end{equation}
where $E_{N-1}(w)$ is given by\\

\noindent $E_{N-1}(w) =$
\begin{equation}
\tilde G_{N-1}'(w) \cdot \alpha L_{N-1}(w) \alpha -\frac{1}{2} \tilde G_{N-1}''(w) \cdot\left( \alpha L_{N-1}(w) \alpha,   \alpha L_{N-1}(w) \alpha\right) -L_{N-1}(w) 
\end{equation}
with 
$$L_{N-1}(w) =\frac{1}{{N-1}} \sum_{j=1}^{N-1} (Y_{N-1}(w) \Psi(w))_{jj},$$
$\Psi$,  $Y_{N-1}$ and  $\tilde G_{N-1}$ being defined in Theorem 5.3 \cite{BC}, 
\eqref{premieroubli} and \eqref{oubli} respectively.
Set \\

$T_N= \frac{1}{2 \sqrt{2}{(N-1)}^2 \sqrt{{N-1}}} \kappa_3(1-\sqrt{-1})$ $$\times  \sum_{i,j,l=1}^{N-1}  \left(Y_{N-1}(w)\right)_{jl}\mathbb{E}\left\{ \alpha (\hat R_{N-1}(w))_{ii}\alpha (\hat R_{N-1}(w))_{ll} \alpha  (\hat R_{N-1}(w))_{ij}\right\},
$$ where $\kappa_3$ still denotes the third cumulant of $\mu$.
Using Cauchy-Schwartz inequality, the bounds \eqref{borneYN}, \eqref{chapeau} and \cite[Lemme 8.1, (8.14)]{BC},      
 it can be  easily proven that 
$$L_{N-1}(w)-T_N =O(1/N).$$
Note moreover that, for any  $m\times m $ matrix $B$ with bounded operator norm
 \begin{eqnarray*}\Tr_m (B T_N) &= &\frac{1}{2 \sqrt{2}{(N-1)}^2 \sqrt{{N-1}}} \kappa_3(1-\sqrt{-1})\\&&\times  \sum_{i,j,l=1}^{N-1}   \Tr_m \mathbb{E}\left\{ \alpha (\hat R_{N-1}(w))_{ii}\alpha (\hat R_{N-1}(w))_{ll} \alpha  (\hat R_{N-1}(w))_{ij}\right. \\ && ~~~~~~~~~~~~~~~~~~ \times \left.\left[\left(B\otimes I_{N-1}\right) Y_{N-1}(w)\right]_{jl}\right\}\\
&= &\frac{1}{2 \sqrt{2}{(N-1)}^2 \sqrt{{N-1}}} \kappa_3(1-\sqrt{-1})\\&&\times  \sum_{i,l=1}^{N-1}   \Tr_m \mathbb{E}\left\{ \alpha (\hat R_{N-1}(w))_{ii}\alpha (\hat R_{N-1}(w))_{ll} \alpha  \right. \\ && ~~~~~~~~~~~~~~~~~~ \times \left.  \left[ \hat R_{N-1}(w) \left( B\otimes I_{N-1}\right) Y_{N-1}(w)\right]_{il}\right\},
\end{eqnarray*}
so that\\

\noindent $\left|\Tr_m (B T_N)\right|$
\begin{eqnarray*} & \leq &\frac{ \vert \kappa_3\vert  m\Vert \alpha\Vert^3}{{2(N-1)} \sqrt{{N-1}}}\\
&&\times  \left\| (\Im w)^{-1}\right\|^2   \mathbb{E}\left\{ \left(\sum_{i,l=1}^{N-1}  \left\| \left[ \hat R_{N-1}(w)\left( B\otimes I_{N-1}\right) Y_{N-1}(w)\right]_{il}\right\|^2\right)^{1/2}\right\}\\ &\leq &\frac{ \vert \kappa_3\vert  m\Vert \alpha\Vert^3\left\| (\Im w)^{-1}\right\|^4 \left\|B \right\|}{{2(N-1)} }
\\&=&O(1/N),\end{eqnarray*}
so that  $$T_N=O(1/{N})$$ and therefore, using \eqref{HTinfiniN}, $$E_{N-1}=O(1/{N}).$$
 Lemma \ref{BC} follows.
\end{proof}
\begin{proposition}\label{propdelta}

~~\\

$\sqrt{N}\left\{ {\rm id}_m \otimes \tr_{N-1}   R_{N-1}(\rho_N e_{11}-\gamma)\1_{\tilde \Omega_{N-1}} \right.$ $$\left. - {\rm id}_m \otimes \phi \left( \left( (\rho_Ne_{11}-\gamma)\otimes 1_{\cal A} -\alpha \otimes x -\beta \otimes a_{N-1}\right)^{-1} \right) \right\}  $$ goes to zero in probability. \end{proposition}
 \begin{proof}
Using  \eqref{sup}, for $N$ large enough, there exists $K>0$ such that $$\left\|  (\rho_N e_{11}-\gamma)\otimes 1_{\cal A}-\alpha\otimes x-\beta \otimes a_{N-1}\right\| \leq K$$
 and on $\tilde \Omega_{N-1}$,  $$\left\| (\rho_N e_{11}-\gamma)\otimes I_{N-1}-\alpha\otimes \frac{W_{N-1}}{\sqrt{N}}-\beta \otimes A_{N-1}\right\| \leq K.$$ Moreover, (see \eqref{HTinfiniNpas} and \eqref{borneRN}), for $N$ large enough,
$$d(0, \text{spect}(\rho_N e_{11} -\gamma)\otimes 1_{\cal A} -\alpha \otimes  x-\beta \otimes a_{N-1}))>C_\epsilon/4$$ and on $\tilde \Omega_{N-1}$
$$d\left(0, \text{spect}\left((\rho_N e_{11} -\gamma)\otimes I_{N-1} -\alpha \otimes  \frac{W_{N-1}}{\sqrt{N}}-\beta \otimes A_{N-1}\right)\right)>C_\epsilon/4.$$
 Let $g: \mathbb{R}\rightarrow \mathbb{R}$ be a ${\cal C}^\infty $ function with  support in $\{C_\epsilon/8\leq \vert x \vert \leq 2K\}$
 and such that  $g\equiv 1$ on $\{C_\epsilon/4\leq \vert x \vert \leq K\}$.
 $f:x\mapsto \frac{g(x)}{x}$ is a ${\cal C}^\infty $ function with compact support.  
 Note that \\

\noindent $\left({\rm id}_{m}\otimes \phi\right) \left(\left((\rho_N e_{11}-\gamma)\otimes 1_{\cal A} -\alpha\otimes x-\beta \otimes a_{N-1} \right)^{-1}\right)$ \begin{equation}\label{fetphi}=\left({\rm id}_{m}\otimes \phi\right) \left(f\left((\rho_N e_{11}-\gamma)\otimes  1_{\cal A} -\alpha\otimes x-\beta \otimes a_{N-1} \right)\right)\end{equation} \noindent and 
on $\tilde \Omega_{N-1}$,
 \begin{equation}\label{zerof} R_{N-1}(\rho_N e_{11}-\gamma)= f\left((\rho_N e_{11}-\gamma)\otimes I_{N-1} -\alpha\otimes \frac{W_{N-1}}{\sqrt{N}}-\beta \otimes A_{N-1} \right). \end{equation}
According to Lemma \ref{BC}, for any $z\in \mathbb{C}\setminus \mathbb{R}$, $$ \sqrt{N}\left( {\rm id}_m\otimes \tr_{N-1}\right)\mathbb{E}\left[ R_{N-1} (\rho_N e_{11}-\gamma -zI_m)\right]$$ \begin{equation}\label{etoile}=
 \sqrt{N} \left({\rm id}_{m}\otimes \phi \right)\left(\left((\rho_N e_{11}-\gamma-zI_{m})\otimes 1_{\cal A} -\alpha\otimes x-\beta \otimes a_{N-1} \right)^{-1}\right) + o^{(z)}(1),\end{equation}
  where there exist polynomials $Q_1$ and $Q_2$ with non negative coefficients and $(d,k) \in \mathbb{N}^2$ such that \begin{equation}\label{doubleetoile}\Vert o^{(z)}(1) \Vert \leq \frac{Q_1(\vert \Im z \vert^{-1})(\vert z \vert +1)^d}{\sqrt{N}} \leq \frac{1}{\sqrt{N}}\frac{Q_2(\vert \Im z\vert)(\vert z \vert +1)^d}{\vert \Im z \vert^k}.\end{equation}
 We recall  Helffer-Sj\"{o}strand's representation  formula : let $f \in C^{k+1}(\mathbb R)$ with compact support and $M$ a Hermitian matrix,
\begin{equation} \label{HS}
 f(M) = \frac{1}{\pi} \int_{\mathbb C} \bar{\partial} F_k(f)(z)\ (M-z)^{-1} d^2z
 \end{equation}
where $d^2 z$ denotes the Lebesgue measure on $\mathbb C$.
\begin{equation} \label{defF_k}
F_k(f)(x+iy) = \sum_{l=0}^k \frac{(iy)^l}{l!} f^{(l)}(x) \chi(y)
\end{equation}
where $\chi : \mathbb R \to \mathbb R^+ $ is a smooth compactly supported function such that $\chi \equiv 1$ in a neighborhood of 0, and $\bar{\partial} =  \partial_x +i \partial_y$. \\
The function $F_k(f)$ coincides with $f$ on the real axis and is an extension to the complex plane. \\
Note that, in a neighborhood of the  real axis, 
\begin{equation} \label{real-axis}
 \bar{\partial} F_k(f)(x+iy) = \frac{(iy)^k}{k!} f^{(k+1)}(x)  = O(|y|^k) \mbox {as } y \rightarrow 0. \end{equation}
 Therefore, by Helffer-Sj\"{o}strand functional calculus,
 $$\sqrt{N}\left( {\rm id}_m\otimes \tr_{N-1}\right) \mathbb{E}\left( f\left((\rho_N e_{11}-\gamma)\otimes I_{N-1} -\alpha\otimes \frac{W_{N-1}}{\sqrt{N}}-\beta \otimes A_{N-1} \right)\right) $$
 $$=\frac{1}{\pi} \int_{\mathbb{C}\setminus \mathbb{R}} \bar \partial F_k(f) (z)\sqrt{N} \left( {\rm id}_m\otimes tr_{N-1}\right) \mathbb{E}\left[ R_{N-1} (\rho_N e_{11}-\gamma-zI_m)\right]  d^2z$$
and\\

\noindent   $\sqrt{N} \left( {\rm id}_{m}\otimes \phi \right)\left[ f\left((\rho_N e_{11}-\gamma)\otimes I -\alpha\otimes x-\beta \otimes a_{N-1} \right)\right]=$
$$\frac{1}{\pi} \int _{\mathbb{C}\setminus \mathbb{R}} \bar \partial F_k(f) (z)\sqrt{N}  \left({\rm id}_m\otimes \phi \right)
\left(\left((\rho_N e_{11}-\gamma -zI_m)\otimes 1_{\cal A} -\alpha\otimes x-\beta \otimes a_{N-1} \right)^{-1}\right)
 d^2z.$$
Hence, using  \eqref{etoile} and \eqref{fetphi}, we can deduce that  $$\sqrt{N}\left( {\rm id}_m\otimes \tr_{N-1}\right) \mathbb{E}\left( f\left((\rho_N e_{11}-\gamma)\otimes I_{N-1} -\alpha\otimes \frac{W_{N-1}}{\sqrt{N}}-\beta \otimes A_{N-1} \right)\right) $$\begin{eqnarray*}&=& \sqrt{N} \left( {\rm id}_{m}\otimes \phi\right) \left(\left((\rho_N e_{11}-\gamma)\otimes I -\alpha\otimes x-\beta \otimes a_{N-1} \right)^{-1}\right)\\&& +  \frac{1}{\pi} \int _{z \in \mathbb{C}\setminus \R} \partial F_k(f) (z)o^{(z)}(1) d^2z. \end{eqnarray*}
 Note that since $f$ and $\chi$ are compactly supported, the last integral is an integral on a bounded set of $\C$ and according to \eqref{doubleetoile} and \eqref{real-axis},
 $$\left\|\frac{1}{\pi} \int_{\mathbb{C}\setminus \mathbb{R}} \partial F_k(f) (z)o^{(z)}(1) d^2z \right\|\leq \frac{C}{\sqrt{N}}.$$
Thus, $$\sqrt{N}\left\{ \mathbb{E}  \left({\rm id}_m\otimes \tr_{N-1}\right) \left( f\left((\rho_N e_{11}-\gamma)\otimes I_{N-1} -\alpha\otimes \frac{W_{N-1}}{\sqrt{N}}-\beta \otimes A_{N-1} \right)\right)\right.$$ \begin{equation}\label{un}\left.- \left({\rm id_m}\otimes \phi\right) \left(\left((\rho_N e_{11}-\gamma)\otimes 1_{\cal A} -\alpha\otimes x-\beta \otimes a_{N-1} \right)^{-1}\right)\right\}  \rightarrow_{N\rightarrow +\infty}0.\end{equation}
Now, we are going to study the  concentration of $$\sqrt{N} \left( {\rm id}_m\otimes \tr_{N-1}\right) \left( f\left((\rho_N e_{11}-\gamma)\otimes I_{N-1} -\alpha\otimes \frac{W_{N-1}}{\sqrt{N}}-\beta \otimes A_{N-1} \right)\right)$$ around its expectation. Define for any $(p,q)\in \{1,\ldots,m\}^2$,
 $h_{pq}: M_{N-1}^{sa}(\mathbb{C}) \rightarrow \mathbb{C}$ by \\

\noindent  $h_{pq}(X)$
$$= \frac{1}{{N-1}}\left( \Tr_{m}\otimes \Tr_{N-1}\right)\left[\left(e_{qp}\otimes I_{N-1}\right)  f\left((\rho_N e_{11}-\gamma)\otimes I_{N-1} -\alpha\otimes X-\beta \otimes A_{N-1}
  \right)\right],$$
	\noindent so that $$ \left({\rm id}_m\otimes \tr_{N-1}\right)\left[ f\left((\rho_N e_{11}-\gamma)\otimes I_{N-1} -\alpha\otimes X-\beta \otimes A_{N-1}
  \right)\right]= \sum_{p,q=1}^m h_{pq}e_{pq}.$$
 Define also $\tilde h_{pq}: \mathbb{R}^{{(N-1)}^2} \rightarrow \mathbb{C}$ by $\tilde h_{pq}=h_{pq}\circ \Psi^{-1}$, where $\Psi$ is defined in \eqref{defiso}. Note that 
 $$\left\| \nabla \tilde h_{pq}(\Psi(X)) \right\|= \left\|\text{grad} h_{pq}(X)\right\|_e.$$
 Applying Poincar\'e inequality for $\tilde h_{pq}$, we get that 
 $$\mathbb{E}\left( \left| h_{pq}(\frac{W_{N-1}}{\sqrt{N}}) -\mathbb{E}(h_{pq}(\frac{W_{N-1}}{\sqrt{N}}))\right|^2 \right) \leq \frac{C}{N} \mathbb{E}\left( \left\| \text{grad} h_{pq}\left(\frac{W_{N-1}}{\sqrt{N}}\right)\right\|_e^2 \right),$$
with
 $$ \left\| \text{grad} h_{pq}(X)\right\|_e^2 =\sup_{w\in S_1(M_{N-1}^{sa}(\mathbb{C}))} \left| \frac{d}{dt} h_{pq}(X+tw)_{\vert_{t=0}} \right|^2.$$

 For $w$ in $S_1(M_{N-1}^{sa}(\mathbb{C}))$, set \begin{eqnarray*}\Delta(t)&=& f\left((\rho_N e_{11}-\gamma)\otimes I_{N-1} -\alpha\otimes ( X+t w)-\beta \otimes A_{N-1} \right) \\&&- f\left((\rho_N e_{11}-\gamma)\otimes I_{N-1} -\alpha\otimes X-\beta \otimes A_{N-1}
  \right)\end{eqnarray*} and $$\Delta(t)=\sum_{p',q'\in \{1,\ldots,m\}^2} e_{p'q'} \otimes \Delta_{p'q'}(t).$$
	Note that $\Delta(t)=\Delta(t)^*$ so that $\Delta_{q'p'}(t)=\Delta_{p'q'}(t)^*$
 We have $$\left| \frac{d}{dt} h_{pq}(X+tw)_{\vert_{t=0}} \right|^2=\left| \lim_{t\rightarrow 0} \frac{1}{t} \tr_{N-1} \Delta_{pq}(t) \right| ^2.$$
Moreover, 
we have \begin{eqnarray*}\left(\Tr_m \otimes \Tr_{N-1}\right) \Delta(t)^2 & = &\sum_{p,q=1}^m \Tr_{N-1} \Delta_{pq}(t) \Delta_{qp}(t)\\& =& \sum_{p,q=1}^m \Tr_{N-1} \Delta_{pq}(t) \Delta_{pq}(t)^*.\end{eqnarray*}
Therefore $ \Tr_{N-1} \Delta_{pq}(t) \Delta_{pq}(t)^* \leq \left(\Tr_m\otimes \Tr_{N-1}\right) \Delta^2(t)$.
Since $f$ is a Lipschitz function on $\mathbb{R}$ with  Lipschitz constant $C_L$, its extension on Hermitian matrices is $C_L$-Lipschitz with respect to the norm $\Vert M\Vert_e=(\Tr_{m(N-1)} M^2)^{1/2}$.
Thus,  \begin{eqnarray*}
\left| \tr_{N-1} \Delta_{pq}(t) \right|^2&\leq&  \tr_{N-1} \Delta_{pq}(t)\Delta_{pq}(t)^*
\\ &\leq & \frac{1}{N-1} \left(\Tr_m\otimes \Tr_{N-1}\right) \Delta(t)^2\\
&\leq & C_L^2  \frac{t^2}{N-1} \left(\Tr_m\otimes \Tr_{N-1}\right) (\alpha^2 \otimes w^2 ) =t^2  \frac{1}{N-1} C_L^2   \Tr_m\alpha^2  .
\end{eqnarray*}
Therefore, 
$$\sup_{w\in S_1(M_{N-1}^{sa}(\mathbb{C}))} \left| \frac{d}{dt} h_{pq}(X+tw)_{\vert_{t=0}} \right|^2 \leq \frac{C}{N},$$  and then 
$$ \mathbb{E}\left( \left| \sqrt{N}\left\{ h_{pq}\left(\frac{W_{N-1}}{\sqrt{N}}\right) -\mathbb{E}\left(h_{pq}\left(\frac{W_{N-1}}{\sqrt{N}}\right)\right)\right\}\right|^2 \right) \leq \frac{C}{N} .$$
It readily follows that $$\sqrt{N}\left({\rm id}_m\otimes \tr_{N-1}\right) \left( f\left((\rho_N e_{11}-\gamma)\otimes I_{N-1} -\alpha\otimes \frac{W_{N-1}}{\sqrt{N}}-\beta \otimes A_{N-1} \right)\right) $$ $$ -   \sqrt{N}\mathbb{E}\left({\rm id}_m\otimes \tr_{N-1}\right) \left( f\left((\rho_N e_{11}-\gamma)\otimes I_{N-1} -\alpha\otimes \frac{W_{N-1}}{\sqrt{N}}-\beta \otimes A_{N-1} \right)\right)$$ \begin{equation}\label{deux}=o_{\mathbb{P}}(1).\end{equation}
Proposition \ref{propdelta} follows from \eqref{zerof}, \eqref{un}, \eqref{deux} and \eqref{smiley}.
 \end{proof}

\section{Proof of Theorem \ref{principal}}\label{Preuve}
According to Lemma \ref{inversible},
$\lambda \in R$ is an eigenvalue of $M_N$ if and only if 
$$\det \left( \lambda e_{11}\otimes I_N -\gamma \otimes I_N -\alpha \otimes \frac{W_N}{\sqrt{N}} -\beta \otimes A_N\right)=0$$
or, since  there exist  permutation matrices  $K_{Nm } $ and $K_{mN}$ in $M_{Nm}$ such that  for any $A \in M_N $ and $B \in M_m$,
\begin{equation}\label{permutation}A \otimes B = K_{Nm}(B\otimes A)K_{mN},\end{equation}
equivalently $$\det \left(I_N \otimes (\lambda e_{11} -\gamma)  - \frac{W_N}{\sqrt{N}} \otimes \alpha   - A_N \otimes \beta \right)=0.$$
Thus, $\lambda$ is an eigenvalue of $M_N$ if and only if \begin{equation}\label{noninjectivite} \exists V \in \mathbb{C}^{Nm}\setminus \{0\}, 
 \left( I_N \otimes (\lambda e_{11}-\gamma) -\frac{W_N}{\sqrt{N}}\otimes \alpha  - A_N \otimes \beta \right)V=0. \end{equation}
Set $$V=\sum_{i=1}^m V_i \otimes e_i $$
where $(e_i)_{i=1,\ldots,m}$ is the canonical basis of $\mathbb{C}^{m}$ and $$V_i= \begin{pmatrix} v_i^{(1)} \in \mathbb{C}\\ V_i^{(2)} \in \mathbb{C}^{N-1} \end{pmatrix}.$$
\eqref{noninjectivite} can be rewritten
$$\sum_{i=1}^m \left\{ \begin{pmatrix} v_i^{(1)}(\lambda e_{11} -\gamma)e_i \\ V_i^{(2)} \otimes (\lambda e_{11} -\gamma)e_i \end{pmatrix}  -\begin{pmatrix} \left(\frac{W_{11}}{\sqrt{N}} v_i^{(1)} +\frac{Y^*}{\sqrt{N}} V_i^{(2)}\right) \alpha e_i \\ \left(\frac{Y}{\sqrt{N}} v_i^{(1)} +\frac{W_{N-1}}{\sqrt{N}} V_i^{(2)} \right) \otimes \alpha e_i \end{pmatrix} \right.$$
$$\left.~~~~~~~~~~~~~~~~~~~~~~~~~-  \begin{pmatrix} \theta v_i^{(1)} \beta e_i\\ A_{N-1} V_i^{(2)} \otimes \beta e_i\end{pmatrix}\right\}=0$$
  which leads to the system $$\left\{ \begin{array}{llr}\left(\lambda e_{11}  -\gamma  -\alpha   \frac{W_{11}}{\sqrt{N}} -\beta \theta \right) \left( 
\sum_{i=1}^m  v_i^{(1)} e_i \right) =  \left(\frac{Y^*}{\sqrt{N}} \otimes \alpha\right) \left( \sum_{i=1}^m  V_i^{(2)} \otimes e_i\right)\\ \begin{array}{ll}
\left(I_{N-1} \otimes (\lambda e_{11}   -\gamma)    -  \frac{W_{N-1}}{\sqrt{N}} \otimes \alpha  - A_{N-1} \otimes \beta \right)
\left( \sum_{i=1}^m  V_i^{(2)}\otimes e_i\right)\\~~~~~~~~~~~~~~~~~~~~= \left(\frac{Y}{\sqrt{N}} \otimes \alpha \right) \left( \sum_{i=1}^m  v_i^{(1)} e_i\right) \end{array}
\end{array} \right.$$
Let $\tau$ be  defined by \eqref{deftau}.
 For any $\lambda \in B(\rho,\tau)$,  according to Lemma \ref{resolvante}
  and \eqref{permutation}, we can define on $\tilde \Omega_{N-1}$
 $$\tilde R_{N-1}(\lambda e_{11}-\gamma) = \left(I_{N-1}\otimes (\lambda e_{11}   -\gamma)    - \frac{W_{N-1}}{\sqrt{N}}\otimes \alpha   - A_{N-1} \otimes \beta  \right)^{-1}.$$
  The following lines hold on $\Omega_N$ (defined by \eqref{defoublie}). \\First, we can deduce from the above system   that  $ \lambda \in B(\rho,\tau)$ is an eigenvalue of $M_N$ if and only if 
  there exists $ (v_i^{(1)})_{i=1,\ldots,m} \in \mathbb{C}^m,
	(V_i^{(2)})_{i=1,\ldots,m} \in \mathbb{C}^{m(N-1)},
	$
	such that: 
	 \begin{equation}\label{equivalence1} 
	\sum_{i=1}^m v_i^{(1)} e_i  \neq 0,\end{equation}
	 \begin{equation}\label{equivalence2} \sum_{i=1}^m  V_i^{(2)}\otimes e_i  = \tilde R_{N-1}(\lambda e_{11}-\gamma)\left( \frac{Y}{\sqrt{N}} \otimes \alpha\right) \left( \sum_{i=1}^m  v_i^{(1)}e_i\right),\end{equation} \begin{equation}\label{equivalence3} \left(\lambda e_{11}  -\gamma  -\alpha   \frac{W_{11}}{\sqrt{N}} -\beta \theta - \frac{1}{N}  \left(Y^* \otimes \alpha \right)\tilde R_{N-1}(\lambda e_{11}-\gamma) \left( Y \otimes \alpha \right)\right)  \left( \sum_{i=1}^m   v_i^{(1)} e_i\right)=0.
  \end{equation}
  Therefore in particular this implies
  \begin{equation}\label{det}\det  \left( X_m(N)\right) =0,\end{equation}
 where $$X_m(N)=  \lambda(N,\rho) e_{11} -\gamma  -\alpha   \frac{W_{11}}{\sqrt{N}} -\beta \theta - \frac{1}{N} \left(Y^*\otimes \alpha\right) \tilde R_{N-1}( \lambda(N,\rho)e_{11}-\gamma) \left( Y \otimes \alpha\right),$$ with $\lambda(N,\rho)$ defined by \eqref{deflambda}.
 Now, noticing that $$ \left(Y^*\otimes \alpha\right) \tilde R_{N-1}( \lambda(N,\rho)e_{11}-\gamma) \left( Y \otimes \alpha\right)$$
 $$= \left(\Tr_{N-1} \otimes {\rm id}_m\right) \left[ (E_{11}\otimes I_m)\left( \tilde Y^* \otimes \alpha\right) \tilde R_{N-1} ( \lambda(N,\rho)e_{11}-\gamma) \left( \tilde  Y \otimes \alpha \right)\right],$$  where $\tilde Y = ( Y \vert 0)  \in M_{N-1}(\mathbb{C})$, and using \eqref{permutation}, it is easy to see that \\

\noindent $ \left(Y^*\otimes \alpha \right)\tilde R_{N-1}( \lambda(N,\rho)e_{11}-\gamma)  \left(Y \otimes \alpha\right)$ $$=  \left(\alpha\otimes Y^*\right) R_{N-1}( \lambda(N,\rho)e_{11}-\gamma) \left(\alpha \otimes Y\right).$$
Let $\rho_N$ be as defined by \eqref{rhoN}. Using 
the identity 
$$R_{N-1}(\rho_N e_{11}-\gamma)-R_{N-1}( \lambda(N,\rho) e_{11}-\gamma) $$ $$=
( \lambda(N,\rho)-\rho_N)R_{N-1}(\rho_N e_{11}-\gamma)\left( e_{11} \otimes I_{N-1}\right) R_{N-1}( \lambda(N,\rho) e_{11}-\gamma),$$
we have 
$$X_m(N)
=H_m(N) + X_m^{(0)}(N),$$
where $$X_m^{(0)}(N)=\omega_m^{(N)} (\rho_N e_{11} -\gamma) -\beta \theta,$$
($\omega_m^{(N)}$ is defined by \eqref{defomegam}),
\begin{eqnarray*}H_m(N)&=& ( \lambda(N,\rho)- \rho_N)e_{11} -\Delta_1(N)-\Delta_2(N) \\ &&+( \lambda(N,\rho)- \rho_N)r_1(N) -\alpha   \frac{W_{11}}{\sqrt{N}} -( \lambda(N,\rho)- \rho_N)^2r_2(N)\end{eqnarray*}
with\\

\noindent  $r_1(N)$ $$= \frac{1}{N} \left(\alpha\otimes Y^*\right) R_{N-1}(\rho_N e_{11}-\gamma)\1_{\tilde \Omega_{N-1}}\left( e_{11} \otimes I_{N-1}\right)  R_{N-1}(\rho_N e_{11}-\gamma)\1_{\tilde \Omega_{N-1}} \left(\alpha \otimes Y\right),$$
\begin{eqnarray*}r_2(N) &=&\frac{1}{N} \left(\alpha\otimes Y^*\right) R_{N-1}(\rho_N e_{11}-\gamma) \1_{\tilde \Omega_{N-1}}\left(e_{11} \otimes I_{N-1} \right) R_{N-1}(\rho_N e_{11}-\gamma)\1_{\tilde \Omega_{N-1}} \\&&~~~~~~~~\times \left(e_{11} \otimes I_{N-1}\right) R_{N-1}( \lambda(N,\rho) e_{11}-\gamma)\1_{\tilde \Omega_{N-1}}\left(\alpha \otimes Y\right),\end{eqnarray*}
\begin{eqnarray*}\Delta_{1}(N)&=&\frac{1}{N}\left( \alpha\otimes Y^*\right) R_{N-1}(\rho_N e_{11}-\gamma)\1_{\tilde \Omega_{N-1}}\left( \alpha \otimes Y\right) \\&&-\alpha \left({\rm id}_m \otimes \tr_{N-1}\right) \left( \left( R_{N-1}(\rho_N e_{11}-\gamma)\1_{\tilde \Omega_{N-1}} \right)\right) \alpha,\end{eqnarray*}
$$\Delta_{2}(N)=\alpha \left({\rm id}_m \otimes \tr_{N-1}\right) \left(  R_{N-1}(\rho_N e_{11}-\gamma)\1_{\tilde \Omega_{N-1}} \right) \alpha $$ $$~~~~~~~~~~~~~~~~~~~~~~~~~~~~~~~~~~~~-\alpha \left({\rm id}_m \otimes \phi\right) \left( \left( (\rho_Ne_{11}-\gamma)\otimes 1_{\cal A} -\alpha \otimes x -\beta \otimes a_{N-1}\right)^{-1} \right) \alpha.$$
\noindent First, we have that, according to  Lemma \ref{loiGN} and using \eqref{borneRN}, 
$$r_1(N) - \alpha \left({\rm id}_m \otimes \tr_{N-1}\right)( R_{N-1}(\rho_N e_{11}-\gamma)\1_{\tilde \Omega_{N-1}} \left(e_{11} \otimes I_{N-1}\right)  R_{N-1}(\rho_N e_{11}-\gamma)\1_{\tilde \Omega_{N-1}})  \alpha$$ $$= o_\mathbb{P}(1).$$
From Lemma \ref{lemR}, almost surely,
$$ \left({\rm id}_m \otimes \tr_{N-1}\right)( R_{N-1}(\rho_N e_{11}-\gamma)\1_{\tilde \Omega_{N-1}}\left( e_{11} \otimes I_{N-1}\right)  R_{N-1}(\rho_N e_{11}-\gamma)\1_{\tilde \Omega_{N-1}} )$$ $$ \vers_{N \rightarrow \infty}\left( {\rm id}_m \otimes \phi\right)( R_\infty(\rho e_{11}-\gamma) \left(e_{11} \otimes 1_\mathcal{A}\right)  R_{\infty}(\rho e_{11}-\gamma)).$$
Therefore, 
\begin{equation}\label{convr1}
r_1(N) \vers^{\mathbb P}_{N \rightarrow \infty} \alpha \left( {\rm id}_m \otimes \phi\right)( R_\infty(\rho e_{11}-\gamma) \left(e_{11} \otimes 1_\mathcal{A} \right) R_{\infty}(\rho e_{11}-\gamma)) \alpha.
\end{equation}
Now,
$$\Vert r_2(N)\Vert  \leq m^2 \Vert \alpha\Vert ^2 \left\| R_{N-1}(\rho_N e_{11}-\gamma)\1_{\tilde \Omega_{N-1}}\right\|^2 \left\| R_N(\lambda(N,\rho) e_{11}-\gamma)\1_{\tilde \Omega_{N-1}}\right\| \frac{\Vert Y \Vert^2}{N}.$$
By the law of large numbers, 
$$  \frac{\Vert Y \Vert^2}{N} = \frac{1}{N} \sum_{j=2}^N |W_{j1}|^2 = 1+ o_{\mathbb{P}}(1).$$
 Moreover, by Lemma \ref{resolvante}, we have 
$$\left\|R_{N-1}(\rho_N e_{11}-\gamma)\1_{\tilde \Omega_{N-1}}\right\| \leq 2/C_\epsilon\; \text{and} \;\left\|R_{N-1}(\lambda(N,\rho) e_{11}-\gamma)\1_{\tilde \Omega_{N-1}}\right\| \leq 2/C_\epsilon.$$ 
Therefore, there exists $C>0$ such that  \begin{equation}\label{convr2} \mathbb{P} \left(\Vert r_2(N)\Vert >C\right) \rightarrow_{N\rightarrow +\infty} 0. \end{equation}
\noindent By Lemma \ref{loiGN}, \begin{equation}\label{delta1zero}\Delta_1(N) =o_{\mathbb{P}}(1).\end{equation}
Now, Proposition \ref{propdelta} readily yields
\begin{equation}\label{Delta2}
\sqrt{N} \Delta_2(N)=o_{\mathbb{P}}(1).
\end{equation}
 \noindent Thus \eqref{convlambda}, \eqref{convrho},  \eqref{convr1}, \eqref{convr2}, \eqref{delta1zero} and  \eqref{Delta2} yield that  \begin{equation} \label{H} H_m(N)=o_{\mathbb{P}}(1).\end{equation}
\noindent Therefore, according to  Lemma \ref{dvptdet} (using \eqref{convrho}, \eqref{HTinfiniNpas} and \eqref{H}), \eqref{det} and \eqref{rhoN}, with a probability going to one as $N$ goes to infinity,   \begin{eqnarray*} 0& = &\det X_m(N)\\
&=& \det(X_m^{(0)}(N)+H_m(N))\\&=& \det (X_m^{(0)}(N)) +\Tr_m \left[B_{X_m^{(0)}(N)} H_m(N)\right]+\epsilon_N\\&=& \Tr_m \left[B_{X_m^{(0)}(N)} H_m(N)\right]+\epsilon_N,
\end{eqnarray*}
where $$B_{X_m^{(0)}(N)}=^t com(X_m^{(0)}(N)),$$
$$\epsilon_N=O(\Vert H_m(N)\Vert^2).$$  Thus, using 
\eqref{convlambda}, \eqref{convrho}, \eqref{convr1}, \eqref{convr2}, \eqref{Delta2} and Proposition \ref{Delta1},  $$\sqrt{N}\epsilon_N =o_{\mathbb{P}}( \sqrt{N} (\lambda-\rho_N)) +o_{\mathbb{P}}(1).$$
Hence, with a probability going to one as $N$ goes to infinity,
$$\sqrt{N} ( \lambda (N,\rho)-\rho_N) \left[ Tr_m  B_{X_m^{(0)}(N)} e_{11}+ Tr_m B_{X_m^{(0)}(N)}r_1(N) +o_{\mathbb{P}}(1)\right]$$ $$
= Tr_m B_{X_m^{(0)}(N)} \sqrt{N} \Delta_1(N) +W_{11} Tr_m B_{X_m^{(0)}(N)} \alpha +o_{\mathbb{P}}(1).$$
Theorem \ref{principal} readily follows from Proposition \ref{Delta1}, the independence of $ \Delta_1(N)$ and $W_{11}$ and the fact that $\omega_m^{(N)}(\rho_N e_{11}-\gamma)$ converges towards $\omega_m(\rho e_{11}-\gamma)$ when $N$ goes to infinity (see 3)  Lemma \ref{lemR}).

\section*{Appendix}
A probability measure $\mu$ satisfies a Poincar\'e inequality if there exists some constant $C_{PI}>0$ such that for any
${\cal C}^1$ function $f: \mathbb{R}\rightarrow \mathbb{C}$  such that $f$ and
$f'$ are in $L^2(\mu)$,
$$\mathbf{V}(f)\leq C_{PI}\int  \vert f' \vert^2 d\mu ,$$
\noindent with $\mathbf{V}(f) = \int \vert
f-\int f d\mu \vert^2 d\mu$. \\
 If the law of a random variable $X$ satisfies the Poincar\'e inequality with constant $C_{PI}$ then, for any fixed $\alpha \neq 0$, the law of $\alpha X$ satisfies the Poincar\'e inequality with constant $\alpha^2 C_{PI}$.\\
Assume that  probability measures $\mu_1,\ldots,\mu_M$ on $\mathbb{R}$ satisfy the Poincar\'e inequality with constant $C_{PI}(1),\ldots,C_{PI}(M)$ respectively. Then the product measure $\mu_1\otimes \cdots \otimes \mu_M$ on $\mathbb{R}^M$ satisfies the Poincar\'e inequality with constant $\displaystyle{C_{PI}^*=\max_{i\in\{1,\ldots,M\}}C_{PI}(i)}$ in the sense that for any differentiable function $f$ such that $f$ and its gradient ${\rm grad} f$ are in $L^2(\mu_1\otimes \cdots \otimes \mu_M)$,
$$\mathbf{V}(f)\leq C_{PI}^* \int \Vert {\rm grad} f \Vert ^2 d\mu_1\otimes \cdots \otimes \mu_M$$
\noindent with $\mathbf{V}(f) = \int \vert
f-\int f d\mu_1\otimes \cdots \otimes \mu_M \vert^2 d\mu_1\otimes \cdots \otimes \mu_M$.
\renewcommand{\thelemma}{A.1}

\begin{lemma}\label{Herbst}{ Lemma 4.4.3 and Exercise 4.4.5 in \cite{AGZ} or Chapter 3 in  \cite{L}. }
Let $\mathbb{P}$ be a probability measure on $\mathbb{R^M}$ which satisfies a Poincar\'e inequality with constant $C_{PI}$. Then there exists $K_1>0$ and $K_2>0$ such that,  for any  Lipschitz function $F$  on $\mathbb{R}^M$ with Lipschitz constant $\vert F \vert_{Lip}$,
$$\forall \epsilon> 0,  \, \mathbb{P}\left( \vert F-\mathbb{E}_{\mathbb{P}}(F) \vert > \epsilon \right) \leq K_1 \exp\left(-K_2\frac{\epsilon}{ \sqrt{C_{PI}} \vert F \vert_{Lip}}\right).$$
\end{lemma}

\renewcommand{\thelemma}{A.2}
\begin{lemma}\label{dvptdet}
Let $A$ and $H$ be $m\times m$ matrices such that, for some $K>0$, \begin{equation}\label{K} \left\|A\right\| \leq K, \; \left\|H\right\| \leq K. \end{equation} Then 
$$ \det (A+H)=  \det (A)+\Tr_m \left( ^t com(A) H \right)+\epsilon$$
where $com(A)$ denotes the comatrix of $A$ and  there exists a constant  $C_{m,K}>0$, only depending on $m$ and $K$,  such that  $\left| \epsilon\right| \leq C_{m,K} \left\| H\right\|^2.$
\end{lemma}	
\begin{proof}
Denote by $a_1, \ldots, a_m$, resp. $h_1, \ldots, h_m$, the columns of the matrix $A$, resp. $H$. 
 Since the determinant of a $m\times m$ matrix is a m-linear function of the $m$ columns, we have
$$\det (A+H) = \det(A) + \sum_{k=1}^m \det ( a_1,\ldots, a_{k-1},h_k, a_{k+1}, \ldots,a_m) +\epsilon,$$
where $\epsilon$ is the sum of  a number  only depending on $m$ of determinants involving at least two columns of $H$.
 Hadamard's inequality and \eqref{K} readily yields that there exists $C_{m,K}>0$ such that $\left| \epsilon\right| \leq C_{m,K} \left\| H\right\|^2.$
Moreover, denoting by $\{e_1, \ldots, e_m\}$ the canonical basis of $\mathbb{C}^m$, we have 
\begin{eqnarray*} \det ( a_1,\ldots, a_{k-1},h_k, a_{k+1}, \ldots,a_m) &=& \sum_{i=1}^m \det ( a_1,\ldots, a_{k-1},H_{ik}e_i, a_{k+1}, \ldots,a_m)\\
&=& \sum_{i=1}^m H_{ik} (com A)_{ik}\\ &=& (^t (com A) H)_{kk}.\end{eqnarray*}
The result readily follows.
\end{proof}
  \noindent {\bf Acknowledgments}\\
	\noindent I am grateful to Serban Belinschi and Catherine Donati-Martin 
for  useful discussions. 
{
\fontsize{9}{10}\selectfont 
}
\end{document}